\documentclass[a4paper]{article}
\usepackage{tikz}
\usepackage{graphicx}
\usepackage{epstopdf}
\usepackage{color}

\usepackage{amsfonts}
\usepackage{amsmath}
\usepackage{amsthm}

\usepackage[margin=1in]{geometry}
\usepackage{setspace}   
\allowdisplaybreaks

\newcommand{\R}{{\bf{R}}}
\newcommand{\I}{{\bf{I}}}
\newcommand{\Ceabe}{\mathrm{C}[a,b]}
\newcommand{\HH}{\cal{H}}
\newcommand{\NN}{\cal{N}}
\newcommand{\FF}{\cal{F}}
\newcommand{\dd}{\textrm{d}}
\newcommand{\Int}{\int\limits}
\newcommand{\Sum}{\sum\limits}
\newcommand{\Sup}{\sup\limits}
\newcommand{\Inf}{\inf\limits}
\newcommand{\Lim}{\lim\limits}

\newcommand{\E}{\textbf{E}}
\newcommand{\Var}{\mathsf{Var}\:}

\newcommand{\Pb}{\textsf{P}}
\newcommand{\Cov}{\textsf{cov}}

\newtheorem{thm}{Theorem}[section]
\newtheorem{cor}{Corollary}[section]
\newtheorem{lem}{Lemma}[section]

\theoremstyle{definition}
\newtheorem{rem}{Remark}[section]

\theoremstyle{definition}
\newtheorem{expl}{Example}[section]

\makeatletter
\renewcommand\@biblabel[1]{}
\makeatother

\begin{document}
\begin{center}
\textbf{ON CROSS-CORRELOGRAM IRF'S ESTIMATORS OF TWO-OUTPUT SYSTEMS \\IN SPACES OF CONTINUOUS FUNCTIONS}\\[4mm]

\textbf{Irina Blazhievska${}^1$ and Vladimir Zaiats${}^2$}\\[4mm]


\noindent ${}^1$ \small{Department of Mathematical Analysis and Probability Theory, \linebreak
NTUU ``Igor Sikorsky Kyiv Polytechnic Institute,'' Av. Peremogy, 37, 03056 Kyiv, Ukraine, \linebreak
ORCID: 	http://orcid.org/0000-0003-4518-5611, e-mail: i.blazhievska@gmail.com} \\[2mm]

\noindent \footnote[2]{Corresponding author.}
\small{Departament d\'{}Economia i d\'{}Hist\`{o}ria Econ\`{o}mica,
\linebreak Universitat Aut\`{o}noma de Barcelona, Edifici B, 08193 Bellaterra (Barcelona), Spain, \linebreak
ORCID: http://orcid.org/0000-0002-2932-4805, e-mail: vladimir.zaiats@gmail.com}
\end{center}
\begin{abstract}
In this paper, single input--double output (SIDO) linear time-invariant (LTI) systems are studied. Both components of system's impulse response function (IRF) are supposed to be real-valued and $L_{2}$-integrable. One component is unknown while the second one is controlled. The problem is to estimate the unknown component after observations of the other component. For this purpose, we apply cross-correlating of the outputs given that the input is a standard Wiener process on $\R$. Weak asymptotic normality of appropriately centred estimators in spaces of continuous functions is proved. This enables us to construct confidence intervals in these spaces. Our results employ techniques related to Gaussian processes and bilinear forms of Gaussian processes.

\end{abstract}

\doublespacing


\section{Introduction}
Identification and estimation of characteristics in stochastic linear systems are usually based on ``black-box'' models. An input may be single or multiple, perfect or noisy, as well as an output may be. The areas of application of these models are quite different, varying from civil engineering (Engberg and Larsen 1995, Chapter 9), (Ren, Zhao and Harik 2004), (Spiridonakos and Chatzi 2014), communications and networks (Demir and Sangiovanni-Vintcentelli 1998, Sections 2.3--2.5), signal and image processing (Camps-Valls, Rojo-\'Alvarez, and Mart\'{\i}nez-Ram\'on 2007), (Ogunfumni 2007, Chapters 3--5), (Gao et al. 2014), (Prabhu 2014, Chapters 7--9), system identification and control (S\"{o}derstr\"{o}m and Stoica 1988, Chapters 5--9), (Sj\"{o}berg et al. 1995), (Chen, Ohlsson and Ljung 2012), applications to biology (Rost, Geske and Baake 2006) and finance (Hatemi-J 2012, 2014). Either in parametric of non-parametric framework, the above mentioned models use the feature of any linear system to be uniquely identified by means of the impulse response function (IRF).

 Our paper deals with a single-input double-output (SIDO) channels model described by a linear time-invariant (LTI) system whose IRF has two real-valued components (kernels) $\{H, g\} \subset L_2(\R)$. System's input is supposed to be a standard Wiener process $W$ on $\R$, whereas the outputs are defined as follows:
 \begin{equation*}
     Y(t)=\Int_{-\infty}^{\infty}H(t-s)dW(s);\qquad X(t)=\Int_{-\infty}^{\infty}g(t-s)dW(s),\ \ t \in \R.
 \end{equation*}

In this paper, we assume that $H$ is unknown while $g$ is known (controlled), and our aim is to estimate~$H$ after observations of the outputs $X$ and $Y$. For a detailed survey of deterministic and statistical approaches to the mentioned problem including correlogram and periodogram methods, we refer the reader to (Blazhievska and Zaiats 2018).
Our main idea is based on the fact that $H$ can be obtained by cross-correlating the $X$ and $Y$ given that $g$ is ``close'' to Dirac's delta:
$$
\Cov(Y(t),X(0))= \Int_{-\infty}^{\infty}H(s)g(s-t)ds=\Big(H\ast g\Big)(t) \approx \Big(H\ast \delta\Big)(t)=H(t),\ \ t\in \R.
$$
This leads us to carrying out the following steps:
\begin{itemize}
\item to introduce the cross-correlogram between the outputs and to take it as an estimator of $H$;
\item to approximate the Dirac delta by a family of $\delta$-like $L_2$-integrable kernels;
\item to use asymptotic results about sample cross-corellograms of Gaussian processes;
\item to study estimator's properties in functional spaces (particularly, in a space of continuous functions).
\end{itemize}
Roughly speaking, $H$ is estimated by the empiric integral-type sample cross-correlogram
$$
\widehat{H}(\tau)=\frac{1}{T}\Int_{0}^{T}Y(t+\tau)X(t)dt,\quad \tau \in \R.
$$
The kernel $g$ is controlled in the following way: it involves a parameter $\Delta$, and we assume that $g$ tends to the Dirac delta as $\Delta\to\infty$. Since the output $X$ is expressed in terms of~$g$, the estimator $\widehat{H}$ depends on two parameters: the averaging length $T$ required to make $\widehat{H}$ asymptotically unbiased, and $\Delta$ whose role consists in scaling. We are interested in establishing a relation between $T$ and $\Delta$ which would provide asymptotic normality of $\widehat{H}$ in the Banach space of continuous functions; we would also like to construct confidence intervals for $H$ in this space. Our basic assumption $H\in L_2(\R)$ will be complemented with some extras related to sample behaviour of Gaussian processes and bilinear forms of Gaussian processes.

 A similar problem was considered in (Li 2005) where the centred estimator $\widehat{H}$ was proved to be asymptotically normal. A recent paper (Blazhievska and Zaiats 2018) has removed extra assumptions on~$H$ reducing them to the one and only condition of $H \in L_2(\R)$. An extension to SIDO-systems has been new. An important feature is that the $L_2$-integrability is, in certain sense, a necessary condition, admitting non-stable system's channels. Recall that stability of an LTI system is related to the fact that $H\in L_1(\R)$.
If $T$ were involved only, it would enable to apply asymptotic properties of Delta matroid integrals generated by functionals based on processes with independent increments, see (Anh, Leonenko, and Sakhno 2007) and (Avram, Leonenko, and Sakhno 2010, 2015).

Our study of weak convergence of a centred estimator $\widehat{H}$ in Banach spaces of continuous functions leads to a situation where two types of processes appear (Blazhievska and Zaiats 2018): a limiting non-stationary Gaussian process  and a pre-limiting process which is stationary and square-Gaussian, 
belonging to an Orlicz space. Almost sure (a.~s.) continuity is proved by means of the entropy approach related to these classes of processes, see (Dudley 1973), (Fernique 1975), (Lifshits 1995, Chapter 14), (Buldygin and Kozachenko 2000, Chapter~4) and (Kozachenko et al. 2017, Chapter~6), while the classical Prokhorov theorem is employed in proving a CLT. Our proof is based on comparison principles for Gaussian processes, similar to what was done in (Buldygin 1983) and extended in (Buldygin and Zaiats 1991, 1995) and (Buldygin and Kharazishvili 2000, Chapter~12). A recent paper (Nourdin, Peccati, and Viens 2014) may open new horizons since it departs from the Gaussian context.

Similar problems for single-input single-output (SISO) systems were considered in (Buldygin, Utzet and Zaiats 2004), (Buldygin and Blazhievska 2009), (Blazhievska 2011), (Kozachenko and Rozora 2015) and (Rozora 2018). For a survey of results on estimation of the correlation function in Gaussian processes which was the parent problem for our setting, see (Buldygin and Kozachenko 2000, Chapter~6).

 The following two circumstances motivate a special role of the functional space we choose:
 \begin{itemize}
   \item it is important to characterise (in terms of the covariance function) continuous Gaussian processes since these processes are mathematical models of numerous real-life phenomena;
   \item the space of real-valued continuous functions is universal in the class of all separable metric spaces.
 \end{itemize}

Assumptions providing that sample paths of different classes of stochastic processes are a.~s. continuous have widely been discussed in the literature. Gaussian processes constitute a theoretically important case. This is why we only state those results that are relevant in the settings. Our study deals with the entropy method. However, in the last twenty years, great progress has been made using majorising measures beginning with the celebrated Talagrand's paper; see (Lifshits 1995, Chapter 16) and (Ledoux and Talagrand 2013, Chapter 11). Note that the entropy method leads to sufficient conditions whereas majorising measures provide necessary and sufficient ones for Gaussian processes to be a.~s. continuous. 

 This paper is organised as follows: Section~2 contains definitions and preliminary information. Section~3 shows when a centred estimator becomes asymptotically Gaussian in spaces of continuous functions. Confidence intervals for the uniform norm of the pre-limiting and limiting processes are constructed in Section~4. Some illustrative examples are given in Section~5. Wiener shot noise processes appearing in our framework were simulated in Wolfram Mathematica$^\circledR$; they are represented on Figures 2--3. Section~6 contains concluding remarks.

The symbols $\Box$, $\triangleleft$ and $\Diamond$ are used to mark the end of a proof, a remark, or an example, respectively.


\section{Definitions and preliminaries}

\paragraph{Notations.} We introduce the following notations to be used throughout all paper:
\begin{itemize}
  \item $L_{p}(\R),$ $p \in [1,\infty),$ is the Banach space of Lebesgue $p$-integrable complex-valued functions $\phi$ on $\R=(-\infty, \infty)$ with the norm $\|\phi\|_{p}=\left(\int_{-\infty}^{\infty}|\phi(x)|^{p}\ dx\right)^{1/p}$;
  \item $L_{\infty}(\R)$ is the Banach space of complex-valued essentially bounded functions $\phi$ on $\R$ with the sup-norm $\|\phi\|_{\infty}=\Sup_{x \in \R}|\phi(x)|$;
  \item $\phi^*$ is the Fourier-Plancherel transform of a given function $\phi \in L_2(\R)$, that is
\begin{equation*}
\phi^{*}(\lambda)=\Int_{-\infty}^{\infty}e^{-i\lambda t}\phi(t)\, dt,\quad\lambda \in \R,
\end{equation*}
implying that $\phi^*\in L_2(\R)$ and $\|\phi^*\|_2=\sqrt{2\pi}\|\phi\|_2$ (e.g., (Kolmogorov and Fomin 1968, 439));
  \item $C[a,b]$,  $[a,b] \subset \R$, is the (separable) Banach space of real-valued continuous functions $\phi$ on $[a,b]$ with the sup-norm $\|\phi\|_{\infty}=\sup_{x \in [a,b]}|\phi(x)|$;
  \item $\textrm{Exp}_{(1)}(\Omega)$ is the Orlicz space of random variables generated by the $\mathbb{C}$-function $U(x)=\exp(|x|)-1$, and defined on a probability space $(\Omega, {\FF}, \Pb)$; (e.g., (Buldygin and Kozachenko 2000, Remark 2.3.1)).
\end{itemize}

\paragraph{Our model and main assumptions.} We consider the following SIDO LTI system with IRF having two real-valued components (or kernels):
\begin{center}
\begin{tikzpicture}[scale=0.5,>=stealth]
\draw[very thick] (-2,-1.5) rectangle (2,1.5);
\filldraw[fill=gray, draw=black, very thick](-2,0) rectangle (2,1.5);
\draw[very thick, ->] (-6,0) -- (-2,0) node[pos=0.5,above] {$W$};
\draw[very thick,black](0, 0.75) node {$H$};
\draw[very thick,black](0, -0.75) node {$g_{\Delta}$};
\draw[very thick, ->] (2,0.75) -- (6,0.75) node[pos=0.5,above] {$Y$};
\draw[very thick, ->] (2,-0.75) -- (6,-0.75) node[pos=0.5,above] {$X_{\Delta}$};
\end{tikzpicture}
\end{center}
where $H$ is an unknown function while $g_{\Delta}$ is a known function dependent on a parameter $\Delta>0$. In the above figure, we shadow the channel where the unknown IRF $H$ appears. Within the paper, we always suppose that the following assumptions hold:
 \begin{itemize}
  \item $H \in L_{2}(\R)$;
  \item the family $(g_{\Delta}, \Delta>0)$ satisfies:
  \begin{equation}\label{f1a}
    g_{\Delta}\in L_{2}(\R);\tag{1a}
    \end{equation}
 \begin{equation}\label{f1b}
     \forall{t\in \R}: g_{\Delta}(t)=g_{\Delta}(-t);\tag{1b}
\end{equation}
\begin{equation}\label{f1c}
     \sup\limits_{\Delta>0}\|{g^{*}_{\Delta}}\|_{\infty}<\infty;\tag{1c}
     \end{equation}
\begin{equation}\label{f1d}
     \exists{c \in (0, \infty)}\ \ \forall{a \in (0, \infty)}: \ \Lim_{\Delta\to\infty}\Sup_{-a\leq \lambda \leq a}\left|g^{*}_{\Delta}(\lambda)-c\right|=0. \tag{1d}
     \end{equation}
 \end{itemize}

 \begin{rem}\label{rem0} In engineering, the Fourier transform of an IRF is often called the frequency transfer function (FTF). Since the basic assumption is $H\in L_2(\R)$, the FTF $H^*$ should be interpreted in the Fourier-Plancherel sense in our framework. \hfill$\triangleleft$
 \end{rem}

 \begin{rem}\label{rem1} It is clear that assumptions (\ref{f1a})--(\ref{f1d}) imply that:
 \begin{itemize}
   \item $g^{*}_{\Delta}$ is a real-valued even function{\rm ;}
   \item if $g_{\Delta}$ is a non-negative function, then\quad $\sup\limits_{\Delta>0}\|{g_{\Delta}}\|_{1}<\infty;$
   \item if\quad $\sup\limits_{\Delta>0}\|{g_{\Delta}}\|_{1}<\infty$, then $g^{*}_{\Delta}$ is the usual Fourier transform of $g_{\Delta};$
   \item the family $(g_{\Delta}, \Delta>0)$ approximates, in a way, the Dirac delta.
   \end{itemize}
Summarising all these statements, the family $(g_{\Delta},\ \Delta>0)$ covers $\delta$-like families for scaled Fourier transforms of classic window functions (Prabhu 2014, Chapter 3). Observe also that the proposed type of $\Delta$-scales for the kernels $g_{\Delta}$ may lead us to continuous wavelets (Najmi 2012, Chapters 4--6). \hfill$\triangleleft$
 \end{rem}

\paragraph{Input and output processes.} The input $W=(W(t),t\in\bf{R})$ to the system is a standard Wiener process on $\bf{R}$. Both outputs are supposed to be observed:
\addtocounter{equation}1
\begin{equation}\label{f2}
   Y(t)=\Int_{-\infty}^{\infty}H(t-s)dW(s); \qquad X_{\Delta}(t)=\Int_{-\infty}^{\infty}g_{\Delta}(t-s)dW(s),\quad t\in\bf{R};
\end{equation}
the integrals in (\ref{f2}) are interpreted as mean-square Riemann integrals. Since $H$ and $g_\Delta$ are $L_2$-integrable and since our system is LTI, the outputs $Y$ and $X_\Delta$ are jointly Gaussian, stationary, zero-mean processes having spectral densities $\frac{1}{2\pi}|H^*(\cdot)|^2$ and $\frac{1}{2\pi}|g^*_{\Delta}(\cdot)|^2$ respectively (e.g., (Lindgren~2006, Theorem~4.8)).

 We also suppose that the process $Y$ and all processes $X_{\Delta}$, $\Delta >0,$ are separable and a.~s. sample continuous on $\R$. This assumption is natural since we deal with Gaussian stationary processes whose correlation functions are continuous; see Belyaev's alternative in (e.g., (Lifshits 1995, Theorem 7.3)).

\paragraph{Cross-correlogram estimators.} We use the following integral-type sample cross-correlogram as an estimator of $H$:
\begin{equation}\label{f3}
\widehat{H}_{T,\Delta}(\tau)=\frac{1}{cT}\Int_{0}^{T}Y(t+\tau)X_{\Delta}(t)dt,\quad \tau \in \bf{R}.
\end{equation}
Here, $c$ is the constant appearing in (\ref{f1d}), $T$ is the length of the interval where the averaging is performed. Since $X_{\Delta}$ and $Y$ are a.~s. sample continuous processes, the integral in (\ref{f3}) may be interpreted in the mean-square Riemann sense. For any $\tau \in \R$, it defines an a.~s. sample continuous function; one has \quad $H(\tau)\neq \E \widehat{H}_{T,\Delta}(\tau)=\frac{1}{c}\Int_{-\infty}^{\infty}g_{\Delta}(s)H(s+\tau)\,ds,$\quad
meaning that $\widehat{H}_{T,\Delta}(\tau)$ is biased as an estimator of $H(\tau)$.
\begin{rem}\label{rem2}
The function $(\mathbf{E} \widehat{H}_{T,\Delta}(\tau),\tau \in\R)$ is non-random and continuous on $\R$ as a normalised (by 1/c) joint covariance function of the processes $X_{\Delta}$ and $Y$, both of which are mean-square continuous. \hfill$\triangleleft$
\end{rem}
\noindent The fact that the estimator $\widehat{H}_{T,\Delta}$ is biased and that it depends on two parameters, $T$ and $\Delta$, enables us to study the role of these parameters in obtaining nice statistical properties of the estimator.

\paragraph{The object of study.} Given that $H\in L_2(\R)$ and given that assumptions (\ref{f1a})--(\ref{f1d}) hold, asymptotic properties of the estimator $\widehat{H}_{T,\Delta}$ are related to the behaviour, as $T\to \infty$ and $\Delta \to\infty$, of the process
 \begin{equation}\label{f4}
 \widehat{Z}_{T,\Delta}(\tau)=\sqrt{T}\Big[\widehat{H}_{T,\Delta}(\tau)-\E \widehat{H}_{T,\Delta}(\tau)\Big],\quad \tau \in \R.
\end{equation}

\begin{rem}\label{rem3}
The process $(\widehat{Z}_{T,\Delta}(\tau), \tau \in \R)$ may be interpreted as a normalised (by $1/\sqrt{T}$) mean-square limit of integral Riemann sums of the type
$$
\Sum_{k}\Big(Y(t_k+\tau)X_{\Delta}(t_k)- \E Y(t_k+\tau)X_{\Delta}(t_k)\Big) \Delta t_k.
$$
The latter allows us to claim that $\widehat{Z}_{T,\Delta}$ is a centred square-Gaussian process and hence it is an $\textrm{Exp}_{(1)}$-process. For further details on Orlicz spaces and pre-ordering of embedded norms, see (Buldygin and Kozachenko 2000, Chapters 1, 4--6). \hfill$\triangleleft$
\end{rem}

\paragraph{Preliminary information.} We give a brief account of the facts we need for our future purposes.
It is clear that process (\ref{f4}) has zero mean; for any $\tau_{1},\tau_{2}\in \R$, the covariance function of $\widehat{Z}_{T,\Delta}$ is as follows:
 \begin{eqnarray}
  \E \widehat{Z}_{T,\Delta}(\tau_{1})\widehat{Z}_{T,\Delta}(\tau_{2}) &=&
   \frac{1}{2\pi c^2}\Int_{-\infty}^{\infty}\Int_{-\infty}^{\infty} \Big[e^{i(\tau_{1}-\tau_{2})\lambda_{1}}|H^{*}(\lambda_{1})|^{2}|g^{*}_{\Delta}(\lambda_{2})|^{2}+ \label{f5} \\[2mm]
  && +e^{i(\tau_{1}\lambda_{1}+\tau_{2}\lambda_{2})}H^{*}(\lambda_{1})H^{*}(\lambda_{2})\overline{g^{*}_{\Delta}(\lambda_{1})g^{*}_{\Delta}(\lambda_{2})}\Big]
       \Phi_{T}(\lambda_{1}-\lambda_{2})\,d\lambda_{1}d\lambda_{2}, \nonumber
 \end{eqnarray}
where $\Phi_{T}$ denotes the well-known Fej\'{e}r kernel, that is,
 $$
 \Phi_{T}(\lambda)=\frac{1}{2\pi T}\left(\frac{\sin(T\lambda /2)}{\lambda /2}\right)^{2},\ \lambda \in \bf{R}.
 $$
In particular, for any  $\tau_{1},\,\tau_{2}\in \R,$ the limit of (\ref{f5}) taken as $T \to \infty$ and $\Delta \to \infty,$ has the form
 \begin{equation}\label{f6}
\lim_{^{T \to \infty}_{\Delta \to \infty}}\E \widehat{Z}_{T,\Delta}(\tau_{1})\widehat{Z}_{T,\Delta}(\tau_{2})
=\frac{1}{2\pi}\Int_{-\infty}^{\infty}\left[e^{i(\tau_{1}-\tau_{2})\lambda}|H^{*}(\lambda)|^{2}+
e^{i(\tau_{1}+\tau_{2})\lambda}(H^{*}(\lambda))^{2}\right]\,d\lambda.
\end{equation}
The function on the right-hand side of (\ref{f6}) will be denoted by $C_\infty(\tau_1,\tau_2)$, $\tau_1,\tau_2\in\R$. Let $Z(\tau),$ $\tau \in \R$, be a measurable separable real-valued Gaussian process with zero mean whose covariance function is $C_{\infty}$:
\begin{equation}\label{f7}
\E Z(\tau_{1})Z(\tau_{2})=C_{\infty}(\tau_{1},\tau_{2}), \quad \tau_1,\tau_2\in\R.
\end{equation}
All finite-dimensional distributions of the process $\widehat{Z}_{T,\Delta}$ converge, as $T\to\infty$ and $\Delta\to\infty$, to those of the process $Z$ (Blazhievska and Zaiats 2018, Theorem 3.2). It is clear that $Z$ is non-stationary.

\paragraph{Further steps to be made.} In this paper, we prove that $\widehat{Z}_{T,\Delta}$ is asymptotically normal in $\Ceabe$. This question is raised in a quite natural way since, for any $T>0$ and $\Delta>0$, the process $\widehat{Z}_{T,\Delta}$ is a.~s. sample continuous on $\R$. Indeed, the process $\widehat{H}_{T,\Delta}$ is a.~s. sample continuous on $\R$, and Remark \ref{rem2} holds. Since $\widehat{Z}_{T,\Delta}$ converges weakly to $Z$ in $\Ceabe$, we can obtain information on asymptotic behaviour of the uniform deviation of the estimator $\widehat{H}_{T,\Delta}(\tau)$ from its mean $\E\widehat{H}_{T,\Delta}(\tau)$ when $\tau$ runs through $[a,b]$. Moreover, since the mean-square error of the (non-stationary) process $Z$ is majorised by that of $Y$ which is stationary, Gaussian comparison inequalities can be applied to these Gaussian processes in order to find simple bounds on the tails of the uniform norm of $Z$.


\section{Asymptotic normality of the process $\widehat{Z}_{T,\Delta}$ in $\Ceabe$}

We recall some useful tools related to Gaussian stochastic processes (Buldygin and Kozachenko 2000, Chapter 4). Let $S$ be a parametric set. A function $\rho(t,s),\, t,s \in S,$ is called a pseudometric on $S$ if it satisfies all axioms of a metric, with the exception that the set $\{(t,s) \in S\times S:\ \rho(t,s)=0\}$ may be wider than the diagonal $\{(t,s) \in S\times S:\ t=s\}$. We write $\NN_{\rho}(S,\varepsilon)$ for the minimum of the number of closed $\rho$-balls of radius $\varepsilon>0$ whose centres lie in S and which cover $S$ ($\varepsilon$-covering). If there is no finite $\varepsilon$-covering of $S$, then we write $\NN_{\rho}(S,\varepsilon)=\infty$. The function ${\NN}_{\rho}(S,\varepsilon), \varepsilon>0,$ is called the metric massiveness ($\varepsilon$-covering number) of $S$ with respect to $\rho$. Further, let $\HH _{\rho}(S,\varepsilon)=\log \NN_{\rho}(S,\varepsilon)$ be a standard notation for the metric entropy of $S$ with respect to $\rho$. For any $\beta>0$, the inequality $\int_{0+}\HH_{\rho}^{\beta}(S, \varepsilon)\dd\varepsilon<\infty$ is always interpreted in the sense that for some (and hence for all) $u>0$ we have $\int_{0}^{u}\HH_{\rho}^{\beta}(S, \varepsilon)\dd\varepsilon<\infty$.
Consider the function
$$
\sigma(\tau)=\left[\int\limits_{-\infty}^{\infty}|H^{*}(\lambda)|^{2}\sin^{2}(\tau\lambda/2)d\lambda\right]^{1/2}, \quad \tau\in \R.
$$
Since $H \in L_{2}(\bf{R})$, it is well-defined and generates the following two pseudometrics: $\sigma(\tau_{1},\tau_{2})=\sigma(|\tau_{1}-\tau_{2}|)$ and $\sqrt{\sigma}(\tau_{1},\tau_{2})=\sqrt{\sigma(\tau_{1},\tau_{2})},\ \tau_{1},\tau_{2} \in \R$. Note that if $H^{*}(\lambda)\neq 0$ for $\lambda$ belonging to a set of a positive Lebesgue measure, then $\sigma$ and $\sqrt{\sigma}$ are metrics.
For all $\varepsilon>0$, put ${\HH}_{\sigma}(\varepsilon)= {\HH}_{\sigma}([0,1],\varepsilon),\,{\HH}_{\sqrt{\sigma}}(\varepsilon)={\HH}_{\sqrt{\sigma}}([0,1],\varepsilon)$. The pseudometrics $\sigma$ and $\sqrt{\sigma}$ depend on $|\tau_{1}-\tau_{2}|$ only. Then for any $[a,b]\subset\R$ and $\beta>0$
 $$
   \int_{0+}{\HH}_{\sigma}^{\beta}(\varepsilon)d\varepsilon<\infty \Longleftrightarrow \int_{0+}{\HH}_{\sigma}^{\beta}([a,b],\varepsilon)d\varepsilon<\infty;\qquad
   \int_{0+}{\HH}_{\sqrt{\sigma}}(\varepsilon)d\varepsilon<\infty \Longleftrightarrow \int_{0+}{\HH}_{\sqrt{\sigma}}([a,b],\varepsilon)d\varepsilon<\infty.
 $$
\begin{rem}\label{rem_SigmaY}
Our choice of the function $\sigma$ is not accidental and is motivated by the following fact:
\begin{equation}\label{f70}
 \E |Y(\tau_2)-Y(\tau_1)|^2=\frac{2}{\pi}\Int_{-\infty}^{\infty}|H^*(\lambda)|^2\sin^2\frac{(\tau_2-\tau_1)\lambda}{2}d\lambda=\frac{2}{\pi}\:\sigma^2(\tau_1,\tau_2), \quad \tau_1, \tau_2 \in \R.
\end{equation}
In other words, the pseudometrics $\sigma$ and $\sqrt{\sigma}$ are related to the mean-square error of the stationary Gaussian output $Y$. \hfill$\triangleleft$
\end{rem}
The next theorem gives sufficient conditions for $\widehat{Z}_{T,\Delta}$ and $Z$ to be a.~s. continuous and shows when $\widehat{Z}_{T,\Delta}$ converges weakly to $Z$ in $\Ceabe$ as $T\to \infty$ and $\Delta \to \infty$. This weak convergence will be denoted by $\widehat{Z}_{T,\Delta}\overset{C[a,b]}\Longrightarrow Z$.
In the sequel, we use the notation $d(\tau_1,\tau_2)=|\tau_1-\tau_2|,$ $\tau_1$, $\tau_2\in \R,$ for the uniform metric.
\begin{thm}\label{thm1}
For any $[a,b]\subset \R$, assume that the following inequality holds\/{\rm :}
\begin{equation}\label{f8}
\int_{0+}\HH_{\sqrt{\sigma}}(\varepsilon)\dd\varepsilon<\infty.
\end{equation}
Then we have{\rm :} \quad {\rm (I)} $Z \in C[a,b]$ a.~s.{\rm ;}\quad {\rm (II)} $\widehat{Z}_{T,\Delta} \in C[a,b]$ a.~s. for any fixed $T>0$ and $\Delta>0;$\quad
{\rm (III)} $\widehat{Z}_{T,\Delta}\overset{C[a,b]}\Longrightarrow Z$ as $T \to \infty$ and $\Delta \to \infty.$
In particular, for any $x>0$
 $$
  \lim_{^{T \to \infty}_{\Delta \to \infty}}{\emph \Pb}\Big\{\sup\limits_{\tau \in [a,b]}\left|\widehat{Z}_{T,\Delta}(\tau)\right|>x\Big\}={\emph \Pb}\Big\{\sup\limits_{\tau \in [a,b]}\left|Z(\tau)\right|>x\Big\}.
 $$
\end{thm}

\begin{rem}\label{rem4}
Statement {\rm (I)} in Theorem \ref{thm1} holds under a weaker assumption than (\ref{f8}), namely,
\begin{equation}\label{f9}
\int_{0+}{\HH}_{\sigma}^{1/2}(\varepsilon)d\varepsilon<\infty.
\end{equation}
Observe (e.g.,(Lifshits 1995, 229)) that if the integral Hunt condition holds for some $\beta>0$:
$$
\int\limits_{-\infty}^{\infty}|H^{*}(\lambda)|^{2}\log^{1+\beta}(1+|\lambda|)\,d\lambda<\infty,
$$
then this condition is sufficient for $Z$ to be a.~s. continuous. \hfill$\triangleleft$
\end{rem}

\begin{rem}\label{rem_DFY}
By equality (\ref{f70}) and by the Fernique theorem (Fernique 1975), the fact that $Y$, being a separable zero-mean stationary Gaussian process on $\R$, is a.~s. continuous implies assumption (\ref{f9}). \hfill$\triangleleft$
\end{rem}

Remarks \ref{rem4}--\ref{rem_DFY} make it clear that if $Y$ is a.~s. continuous, then the limit $Z$ is also a.~s. continuous.

\begin{rem}\label{rem5}
Assumption (\ref{f8}) is satisfied (Buldygin, Utzet, and Zaiats 2004) if the following condition holds for some $\beta>0$:
$$
\int\limits_{-\infty}^{\infty}|H^{*}(\lambda)|^{2}\log^{4+\beta}(1+|\lambda|)\,d\lambda<\infty.
$$
\end{rem}
\vspace{-1.5\baselineskip}
\hfill$\triangleleft$

\vspace{.5\baselineskip}
The proof of Theorem \ref{thm1} requires auxiliary results. First of all, consider the link between the pseudometrics $\sigma$ and $\sqrt{\sigma}$. Since, for all $\tau_{1},\tau_{2}\in \R$,
\begin{equation}\label{f10}
\sigma(\tau_{1},\tau_{2})\leq\left[\max\limits_{\tau_{1},\tau_{2}\in \R}\sigma(\tau_{1},\tau_{2})\right]^{1/2}\sqrt{\sigma}(\tau_{1},\tau_{2})\leq\|H^{*}\|_{2}^{1/2}\sqrt{\sigma}(\tau_{1},\tau_{2}),
\end{equation}
condition (\ref{f8}) yields $\int_{0+}{\HH}_{\sigma}(\varepsilon)d\varepsilon<\infty$ implying (\ref{f9}).
For all $T>0,\,\Delta>0,$ we introduce a new family of pseudometrics
 \begin{equation}\label{F:rho_TDelta}
\rho_{(T,\Delta)}(\tau_{1},\tau_{2})=\left(\E|\widehat{Z}_{T,\Delta}(\tau_{2})-\widehat{Z}_{T,\Delta}(\tau_{1})|^{2}\right)^{1/2},\quad \tau_{1},\tau_{2} \in \R.
\end{equation}
\begin{lem}\label{lem1}
For any $T>0,$ $\Delta>0,$ and any $\tau_{1},\tau_{2}\in \R,$ the following inequality holds\/{\rm :}
\begin{equation}\label{f11}
\rho_{(T,\Delta)}(\tau_{1},\tau_{2}) \leq \frac{1}{c}\: \left(\frac{2}{\pi}\:\|H^{*}\|_{2}\right)^{1/2} \Sup_{\Delta>0}\|g^*_{\Delta}\|_{\infty}\cdot\sqrt{\sigma}(\tau_{1},\tau_{2}).
\end{equation}
Moreover, the pseudometric $\rho_{(T,\Delta)}$ is continuous with respect to the pseudometic $\sigma$.
\end{lem}

\begin{proof}
Apply the Cauchy-Schwarz inequality, the Young inequality for convolutions (e.g., (Edwards 1965, 655)), and the fact that $\|\Phi_{T}\|_{1}=1$, to (\ref{f5}). Then
 \begin{eqnarray*}
  \lefteqn{\rho_{(T,\Delta)}^{2}(\tau_{1},\tau_{2})=\frac{1}{\pi c^2}\Bigg|\Int_{-\infty}^{\infty}\Int_{-\infty}^{\infty}|H^{*}(\lambda_{1})|^{2}|g^*_\Delta(\lambda_2)|^2\Phi_{T}(\lambda_{1}-\lambda_{2})\cdot
   \Big[\sin\frac{(\tau_{1}-\tau_{2})\lambda_{1}}{2}\Big]^{2}d\lambda_{1} d\lambda_{2}+}\\[2mm]
 && +\frac{1}{2}\Int_{-\infty}^{\infty}\Int_{-\infty}^{\infty} e^{i(\tau_{1}\lambda_{1}+\tau_{2}\lambda_{2})} \Big[e^{i(\lambda_{1}(\tau_{2}-\tau_{1})}-2+e^{i\lambda_{2}(\tau_{1}-\tau_{2})}\Big]\times\\[2mm]
 && \times H^{*}(\lambda_{1})H^{*}(\lambda_{2}) \overline{g^*_\Delta(\lambda_1)g^*_\Delta(\lambda_2)}\Phi_{T}(\lambda_{1}-\lambda_{2})d \lambda_{1} d\lambda_{2}\Bigg|\leq\\[2mm]
 &\leq&
   \frac{1}{\pi c^2}\Big(\Sup_{\Delta>0}\|g^*_{\Delta}\|_{\infty}\Big)^{2}  \Int_{-\infty}^{\infty}|H^{*}(\lambda_{1})|^{2}\Big[\sin\frac{(\tau_{1}-\tau_{2})\lambda_{1}}{2}\Big]^{2}d\lambda_{1}+\\[2mm]
 && +\frac{1}{\pi c^2}\Big(\Sup_{\Delta>0}\|g^*_{\Delta}\|_{\infty}\Big)^{2}\Int_{-\infty}^{\infty} |H^{*}(\lambda_{1})| \Big[|H^{*}|\ast \Phi_{T}\Big](\lambda_{1})\cdot \Big|\sin\frac{(\tau_{1}-\tau_{2})\lambda_{1}}{2}\Big| d\lambda_{1}\leq\\[2mm]
 &\leq&
     \frac{1}{\pi c^2}\Big(\Sup_{\Delta>0}\|g^*_{\Delta}\|_{\infty}\Big)^{2}\Bigg[\Int_{-\infty}^{\infty}|H^{*}(\lambda_{1})|^{2}\Big[\sin\frac{(\tau_{1}-\tau_{2})\lambda_{1}}{2}\Big]^{2}d\lambda_{1}+\\[2mm]
 && +\Bigg[\Int_{-\infty}^{\infty} |H^{*}(\lambda_{2})|^{2}\Big[\sin\frac{(\tau_{1}-\tau_{2})\lambda_{2}}{2}\Big]^{2} d\lambda_{2}\Bigg]^{\frac{1}{2}}\Big\||H^{*}|\ast\Phi_{T}\Big\|_{2}\Bigg]\leq\\[2mm]
 &\leq&
   \frac{1}{\pi c^2}\Big(\Sup_{\Delta>0}\|g^*_{\Delta}\|_{\infty}\Big)^{2}\Bigg[\sigma^2(\tau_{1},\tau_{2})+\sigma(\tau_{1},\tau_{2})\|H^*\|_2\|\Phi_{T}\|_1\Bigg]\leq
   \frac{2}{\pi c^2}\Big(\Sup_{\Delta>0}\|g^*_{\Delta}\|_{\infty}\Big)^{2}\|H^*\|_2\cdot\sigma(\tau_{1},\tau_{2})
 \end{eqnarray*}
implying (\ref{f11}). Note that the inequality is uniform in $T>0$ and $\Delta>0$.
By the Lebesgue dominated convergence, $\rho_{(T,\Delta)}(\tau_{1}, \tau_{2}) \to 0$  as $\sqrt{\sigma}(\tau_{1},\tau_{2}) \to 0$. Therefore $\rho_{(T,\Delta)}$ is continuous with respect to $\sqrt{\sigma}$.  It is clear that convergence to zero is equivalent both for $\sqrt{\sigma}$ and~$\sigma$. This proves Lemma \ref{lem1}.
\end{proof}

\begin{proof}[Proof of Theorem \ref{thm1}] By the Dudley theorem
on continuity of Gaussian processes (Dudley 1973), Statement (I) in Theorem \ref{thm1} holds if, for any $[a,b]\subset\R$,
\begin{equation}\label{f12}
\int_{0+}{\HH}_{d_{Z}}^{1/2}([a,b],\varepsilon)d\varepsilon<\infty,
\end{equation}
where $d_{Z}(\tau_{1},\tau_{2})=\left(\E|Z(\tau_{2})-Z(\tau_{1})|^{2}\right)^{1/2}$.
By the Cauchy-Schwarz inequality applied to (\ref{f6}), we obtain
\begin{equation}\label{f120}
d_{Z}^{2}(\tau_{1},\tau_{2})\leq \frac{4}{\pi}\int\limits_{-\infty}^{\infty}|H^{*}(\lambda)|^{2}\left|\sin\frac{(\tau_{2}-\tau_{1})\lambda}{2}\right|^2\,d\lambda =\frac{4}{\pi}\:\sigma^2(\tau_{1},\tau_{2}),\quad \tau_{1},\tau_{2}\in\R.
\end{equation}

The Lebesgue dominated convergence implies that $\sigma(\tau_{1},\tau_{2}) \to 0$ as $d(\tau_{1}, \tau_{2})\to 0$. Therefore $Z$ is a mean-square continuous process. Formula (\ref{f12}) holds if $\int_{0+}{\HH}_{\sigma}^{1/2}(\varepsilon)\,d\varepsilon<\infty$. The proof of Statement (I) in Theorem \ref{thm1} becomes complete by application of (\ref{f8}).

The next step  is based on the fact that $\widehat{Z}_{T,\Delta}$ is a square-Gaussian process. By Theorem 3.5.4 in (Buldygin and Kozachenko 2000), Statement (II) holds if, for any $[a,b]\subset\R,$
\begin{equation}\label{f13}
\int_{0+}{\HH}_{\rho_{(T,\Delta)}}([a,b],\varepsilon)\,d\varepsilon<\infty.
\end{equation}
Lemma \ref{lem1} stands that $\rho_{(T,\Delta)}$ is majorised by $\sqrt{\sigma}$ which is continuous with respect to $d$. Therefore condition (\ref{f8}) yields (\ref{f13}) proving Statement (II) in Theorem \ref{thm1}.

The last step of the proof is based on application of relevant results on square-Gaussian processes. We will use Lemma 4.2.1 in (Buldygin and Kozachenko 2000) whose adapted version is as follows:

\begin{lem}\label{lemBK_2000}
Let $\widehat{Z}_{T,\Delta}$ be a family of a.~s. sample continuous random processes. For any $[a,b]\subset\R$, assume that the following conditions hold\/$:$
\begin{itemize}
  \item [{\rm (i)}] $\sup\limits_{T,\Delta>0}\sup\limits_{\tau_{1},\tau_{2}\in [a,b]}{\emph \E}\exp\left\{\frac{|\widehat{Z}_{T,\Delta}(\tau_{2})-\widehat{Z}_{T,\Delta}(\tau_{1})|}{\sqrt{8}\,\rho_{(T,\Delta)}(\tau_{1},\tau_{2})}\right\}<\infty;$
  \item [{\rm (ii)}] the pseudometric $\rho_{\infty}(\tau_{1},\tau_{2})=\sup\limits_{T,\Delta>0}\rho_{(T,\Delta)}(\tau_{1},\tau_{2})$ is continuous with respect to $d$;
  \item [{\rm (iii)}] $\lim\limits_{u\downarrow 0}\sup\limits_{T,\Delta>0}\int_{0}^{u}{\HH}_{\rho_{(T,\Delta)}}^{1/2}([a,b],\varepsilon)\,d\varepsilon=0.$
\end{itemize}
Then, for any $\delta>0,$ we have$:$\quad
$
\lim\limits_{h\downarrow 0}\sup\limits_{T,\Delta>0}\textsf{\emph{P}}\left\{\displaystyle{\sup_{\substack{\tau_{1},\tau_{2} \in [a,b]\\|\tau_{2}-\tau_{1}|<h}}}|\widehat{Z}_{T,\Delta}(\tau_{2})-\widehat{Z}_{T,\Delta}(\tau_{1})|>\delta\right\}=0.
$
\end{lem}
Let us check whether the assumptions of Lemma \ref{lemBK_2000} hold. Since $\widehat{Z}_{T,\Delta}$ is a square-Gaussian zero-mean process, we have, by Theorem 6.2.2 in (Buldygin and Kozachenko 2000), for any $T>0$ and $\Delta>0$:
\begin{equation}\label{f14}
\sup\limits_{T,\Delta>0}\sup\limits_{\tau_{1},\tau_{2}\in \R}\E\exp\left\{\frac{|\widehat{Z}_{T,\Delta}(\tau_{2})-\widehat{Z}_{T,\Delta}(\tau_{1})|}{\sqrt{8}\,\rho_{(T,\Delta)}(\tau_{1},\tau_{2})}\right\}<\infty.
\end{equation}
 The Lebesgue dominated convergence implies that the pseudometric $\sqrt{\sigma}$ is continuous with respect to the metric~$d$. By Lemma \ref{lem1}, the pseudometric
\begin{equation}\label{f15}
\rho_{\infty}(\tau_{1},\tau_{2})=\sup\limits_{T,\Delta>0}\rho_{(T,\Delta)}(\tau_{1},\tau_{2}), \quad \tau_{1},\tau_{2} \in \R,
\end{equation}
is continuous with respect to $d$.
By inequality (\ref{f11}), condition (\ref{f8}) yields that for any $[a,b]\subset\R$
\begin{equation}\label{f16}
\lim\limits_{u\downarrow 0}\sup\limits_{T,\Delta>0}\int_{0}^{u}{\HH}_{\rho_{(T,\Delta)}}^{1/2}([a,b],\varepsilon)\,d\varepsilon\leq \lim\limits_{u\downarrow 0}\int_{0}^{u}{\HH}_{\rho_{\infty}}^{1/2}([a,b],\varepsilon)\,d\varepsilon=0.
\end{equation}
Formulas (\ref{f14})--(\ref{f16}) show that conditions of Lemma \ref{lemBK_2000} hold. Therefore, for any $\delta>0$ and any $[a,b]\subset\R$,
\begin{equation}\label{f17}
\lim\limits_{h\downarrow 0}\sup\limits_{T,\Delta>0}\textsf{P}\left\{\displaystyle{\sup_{\substack{\tau_{1},\tau_{2} \in [a,b]\\|\tau_{2}-\tau_{1}|<h}}}|\widehat{Z}_{T,\Delta}(\tau_{2})-\widehat{Z}_{T,\Delta}(\tau_{1})|>\delta\right\}=0.
\end{equation}
Theorem 3.2 in (Blazhievska and Zaiats 2018) 
combined with (\ref{f17}) implies that the
Prohorov theorem on weak convergence of
stochastic processes in $C[a,b]$ holds (e.g., (Billingsley 2013, 59)). This proves Statement
(III) in Theorem \ref{thm1}. In this way, the proof of Theorem \ref{thm1} is complete.
\end{proof}


\section{Confidence intervals for the uniform norms of $\widehat{Z}_{T,\Delta}$ and $Z$}

Theorem \ref{thm1} states a CLT in $\Ceabe$ for the centred estimator $\widehat{H}_{T,\Delta}$. In particular, for any $x>0$
 \begin{eqnarray*}
  \lefteqn{\lim_{^{T \to \infty}_{\Delta\to\infty}}\Pb \Big\{\Sup_{\tau \in [a,b]}\Big|\sqrt{T}\big(\widehat{H}_{T,\Delta}(\tau)-\E \widehat{H}_{T,\Delta}(\tau)\big)\Big|>x\Big\}=} \\[2mm]
  && =\lim_{^{T \to \infty}_{\Delta\to\infty}}\Pb \Big\{\Sup_{\tau \in [a,b]}\Big|\widehat{Z}_{T,\Delta}(\tau)\Big|>x\Big\}=\Pb \Big\{\Sup_{\tau \in [a,b]}\Big|Z(\tau)\Big|>x\Big\}.
 \end{eqnarray*}
This equality enables us to construct confidence intervals for $\widehat{Z}_{T,\Delta}$ and $Z$ which are accurate enough for large $T$ and $\Delta$. Observe that each of the parameters $T$ and $\Delta$ tends to infinity on its own; a link between them will appear latter on, when we look at the error term.
We use two different techniques for studying tails of the uniform norm for the underlying processes:
\begin{itemize}
  \item the pre-limiting process $\widehat{Z}_{T,\Delta}$ requires a theory related to square-Gaussian  processes (or $\textrm{Exp}_{(1)}$-processes), see (Buldygin and Kozachenko 2000, Chapters 4--6) and (Kozachenko and Rozora 2015);
  \item the limiting process $Z$ is treated by means of comparison principles related to Gaussian processes, see (Buldygin 1983) and (Buldygin and Kharazishvili 2000, Chapter 12).
\end{itemize}
Note that each of these techniques requires the mentioned zero-mean processes to be separable and to satisfy certain majorisation between their entropy characteristics or their correlation functions.

\subsection{Bounds on large deviations of the supremum of $Z$}\label{subsection1}
Since the distribution of a Gaussian process is completely determined by the covariance function of this process, it is natural to anticipate that a majorisation of covariances would lead to a certain subordination between the processes. The well-known Anderson and Slepian inequalities, see Theorem~3 and Lemma~3 respectively in (Buldygin and Kharazishvili 2000, Chapter~12), turn out to be technical tools for obtaining useful results on comparison of Gaussian processes. Let us focus on them in more detail.

Note that (\ref{f70}) and (\ref{f120}) imply the following majorisation between the mean-square errors of $Z$ and $Y$:
\begin{equation}\label{fpo}
\E|Z(\tau_2)-Z(\tau_1)|^2 \leq\frac{4}{\pi}\Int_{-\infty}^{\infty}|H^*(\lambda)|^2 \sin^2{\frac{(\tau_2-\tau_1)\lambda}{2}} d\lambda=2\E|Y(\tau_2)-Y(\tau_1)|^2,  \quad \tau_1, \tau_2 \in \R.
\end{equation}
Recall that $Z$ and $Y$ are separable, Gaussian and have zero means. We will compare the limiting process~$Z$ to a new process obtained by adding an independent random variable to the output $Y$ scaled by $\sqrt{2}$. This idea enables us to use  Slepian's and Anderson's inequalities simultaneously.

\begin{rem}\label{rem_ZYTDelta}
Since formula (\ref{fpo}) gives majorisation of  $Z$ in terms of the output $Y$ only, where neither of the parameters $T$ and $\Delta$ is involved, the confidence intervals for $Z$ constructed by comparison with $Y$ will be free of these parameters. \hfill$\triangleleft$
\end{rem}

Let us give an adapted version of Theorem 4 in (Buldygin and Kharazishvili 2000, Chapter 12). The proof is dropped; it follows the lines of the source theorem.

\begin{thm}\label{thm2}
Assume that $\xi$ is an $N(0,1)$-distributed random variable independent of $Y$ and let
$$
b^{2}(\tau)=\Sup_{\tau \in [a,b]}\Bigg(\max\Big(0,\ \emph{\E}(\sqrt{2}Y)^2(\tau)-\emph{\E} Z^2(\tau)\Big)\Bigg)-\Big(\emph{\E} (\sqrt{2}Y)^2(\tau)-\emph{\E} Z^2(\tau)\Big), \quad \tau \in [a,b].
$$
Then, for any  $x>0$,  we have
\begin{equation}\label{fpb}
{\emph \Pb} \Bigg\{\Sup_{\tau \in [a,b]}\Big|Z(\tau)\Big|>x\Bigg\}\leq 2 {\emph \Pb} \Bigg\{\Sup_{\tau \in [a,b]}\Big(\sqrt{2}Y(\tau)+\xi b(\tau)\Big)>x\Bigg\}.
\end{equation}
\end{thm}
\noindent We will specify the structure of the function $\big(b(\tau),\tau \in [a,b]\big)$ below. Since
$$
\E (\sqrt{2}Y)^2(\tau)-\E Z^2(\tau)=\frac{1}{2\pi}\Int_{-\infty}^{\infty}|H^*(\lambda)|^2 d\lambda-\frac{1}{2\pi}\Int_{-\infty}^{\infty}e^{i(2\tau)\lambda}(H^*(\lambda))^2 d\lambda=\|H\|^2_2-\Int_{-\infty}^{\infty}H(2\tau-s)H(s)ds,
$$
and since \quad
$
\Int_{-\infty}^{\infty}H(2\tau-s)H(s)ds\leq\Big|\Int_{-\infty}^{\infty}H(2\tau-s)H(s)ds\Big|\leq\|H\|^2_2
$\quad for any $\tau \in [a,b]$,
one has
$$
\max\left(0,\ \E(\sqrt{2}Y)^2(\tau)-\E Z^2(\tau)\right)=\|H\|^2_2-\Int_{-\infty}^{\infty}H(2\tau-s)H(s)ds.
$$
Therefore
\begin{equation}\label{F:b}
b^2(\tau)=\Int_{-\infty}^{\infty}H(2\tau-s)H(s)ds-\Inf_{\tau \in [a,b]}\Big(\Int_{-\infty}^{\infty}H(2\tau-s)H(s)ds\Big), \quad \tau \in [a,b].
\end{equation}
Note that all these calculations are only based on the assumption that $H$ is $L_2$-integrable.

Inequality (\ref{fpb}) in Theorem \ref{thm2} may be rewritten in another way. The following statements can be extracted from (Buldygin and Kharazishvili 2000, Chapter~12); Corollary \ref{cor1pb} is contained in the proof of Theorem~3, and Corollary~\ref{cor2pb} is adapted from Lemma~4. We drop explicit proofs of these corollaries.

\begin{cor}\label{cor1pb}
For any $x>0$, we have
\begin{equation*}\label{fpb11}
\emph{\Pb} \Bigg\{\Sup_{\tau \in [a,b]}\Big|Z(\tau)\Big|>x\Bigg\}\leq2\emph{\Pb} \Bigg\{\Sup_{\tau \in [a,b]}Y(\tau)>\frac{\gamma x}{\sqrt{2}}\Bigg\}+2\emph{\Pb} \Bigg\{\xi \cdot\Sup_{\tau \in [a,b]}b(\tau)>(1-\gamma)x\Bigg\}.
\end{equation*}
Here, $\gamma$ is any number belonging to $[0,1]$ and $b$ is as defined in (\ref{F:b}).
\end{cor}

\begin{cor}\label{cor2pb}
For any $x>0$, we have
\begin{equation*}\label{f12pb}
\emph{\Pb} \Bigg\{\Sup_{\tau \in [a,b]}\Big|Z(\tau)\Big|>x\Bigg\}\leq 2\emph{\Pb} \Bigg\{\Sup_{\tau \in [a,b]}|Y(\tau)|>\frac{x}{2\sqrt{2}}\Bigg\}+4 \exp\Bigg\{-\frac{x^{2}}{B_{[a,b]}}\Bigg\},
\end{equation*}
where\quad
$
B_{[a,b]}=16\Sup_{\tau\in [a,b]}\Big|{\emph\E} (\sqrt{2}Y)^2(\tau)-{\emph\E}  Z^2(\tau)\Big|=16\|H\|^2_2-16\cdot\Inf_{\tau \in [a,b]}\Big(\Int_{-\infty}^{\infty}H(2\tau-s)H(s)ds\Big).
$
\end{cor}

\noindent Corollaries \ref{cor1pb}--\ref{cor2pb} show a way to estimate the tails of (non-stationary) $Z$  through those of (stationary) $Y$. General results on behaviour of the tails of Gaussian stationary processes may be found in classical sources such as (Cram\'er and Leadbetter 1967, Chapters 9--11), (Marcus and Shepp 1972), (Fernique
1975), (Adler 2000, Chapter V), as well as in (Buldygin 1983) and (Lifshits 1995, Sections 12--13).


\subsection{Inequalities for large deviations of $\widehat{Z}_{T,\Delta}$}\label{subsection2}
Our next step is to study centred square-Gaussian processes belonging to the Orlicz space $\textrm{Exp}_{(1)}(\Omega)$. In the section, we focus on constructing upper bounds for the distribution of the separable square-Gaussian process $\widehat{Z}_{T,\Delta}$ and that of its supremum.
We will use technical tools developed in (Buldygin and Kozachenko 2000, Chapters 4--6), (Kozachenko and Rozora 2015) and (Kozachenko et al. 2017, Chapter 6).
The next result gives point-wise confidence intervals for the centred estimator $\widehat{H}_{T,\Delta}$.
\begin{thm}\label{thm3}
For all $\tau \in \R$ and $x>0$, we have
 \begin{eqnarray}
  && \Pb \Big\{\widehat{H}_{T,\Delta}(\tau)-\E\widehat{H}_{T,\Delta}(\tau)\geq x\:\sqrt{\Var \widehat{H}_{T,\Delta}(\tau)}  \Big\}\leq K(x),\nonumber \\[2mm]
  && \Pb \Big\{\widehat{H}_{T,\Delta}(\tau)-\E\widehat{H}_{T,\Delta}(\tau)\leq -x\:\sqrt{\Var \widehat{H}_{T,\Delta}(\tau)} \Big\}\leq K(x), \label{fpbZ1}\\[2mm]
  && \Pb \Big\{|\widehat{H}_{T,\Delta}(\tau)-\E\widehat{H}_{T,\Delta}(\tau)|\geq x\:\sqrt{\Var \widehat{H}_{T,\Delta}(\tau)} \Big\}\leq 2K(x),\nonumber
 \end{eqnarray}
where\/ $\Var\widehat{H}_{T,\Delta}(\tau)=\E|\widehat{H}_{T,\Delta}(\tau)-\E\widehat{H}_{T,\Delta}(\tau)|^2>0$ and
$
K(x)=(1+\sqrt{2}x)^{1/2}\exp{\Big(-x/\sqrt{2}\Big)}, x>0.
$
\end{thm}
\begin{proof}
By Lemma 6.3.1 and by Theorem 6.2.2 in (Buldygin and Kozachenko 2000) applied to the square-Gaussian process $\widehat{H}_{T,\Delta}-\E \widehat{H}_{T,\Delta}$,  the following inequality holds for $|s|<1$:
$$
\E\exp{\Bigg\{\frac{s}{\sqrt{2}}\Bigg(\frac{\widehat{H}_{T,\Delta}(\tau)-\E\widehat{H}_{T,\Delta}(\tau)}{\sqrt{\Var\widehat{H}_{T,\Delta}(\tau)}}\Bigg)\Bigg\}}\leq\frac{1}{\sqrt{1-|s|}}\exp{\Big(-|s|/2\Big)}.
$$
This formula combined with the Chebyshev-Markov inequality implies that
$$
\Pb\Bigg\{\frac{\widehat{H}_{T,\Delta}(\tau)-\E\widehat{H}_{T,\Delta}(\tau)}{\sqrt{\Var\widehat{H}_{T,\Delta}(\tau)}}\geq x\Bigg\}\leq\frac{1}{\sqrt{1-s}}\exp{\Big(-\frac{s}{2}-\frac{s}{\sqrt{2}}\:x\Big)}
$$
for any $s\in (0,1)$ and any $x>0$. The first inequality in (\ref{fpbZ1}) is obtained by minimising the right-hand side of the above inequality with respect to $s\in (0,1)$. The second inequality is proved along similar lines, and the third one is a corollary of the first two inequalities. 
\end{proof}

\begin{rem}\label{rem_ZTDelta}
The role of the parameters $T$ and $\Delta$ in construction of confidence intervals is as follows:
 \begin{itemize}
   \item $T$ is responsible for proving CLT and determines the number of approximations; confidence intervals are constructed based on this parameter only.
   \item $\Delta$ is required in order to make $\widehat{H}_{T,\Delta}$ asymptotically unbiased; it appears in the structure of all characteristics related to pseudometrics and entropies.
 \end{itemize}
  The balance between $T$ and $\Delta$ will emerge in the study of error term's behaviour. \hfill$\triangleleft$
\end{rem}

Take $\tau\in\R$. If the function $\Var\widehat{H}_{T,\Delta} (\tau),$ $T>0,$ $\Delta>0,$ is known, or if one disposes of an upper bound on this function, then (\ref{fpbZ1}) enables us to obtain interval estimates for $\E \widehat{H}_{T,\Delta}(\tau).$ Note that $\Var\widehat{H}_{T,\Delta}(\tau)=T^{-1}\E|\widehat{Z}_{T,\Delta}(\tau)|^2,$ $T>0$; see (\ref{f4}). Therefore, for all $\tau\in \R$ and for any $x>0$,
$$
\Pb \Bigg\{\Big|\widehat{H}_{T,\Delta}(\tau)-\E\widehat{H}_{T,\Delta}(\tau)\Big|\geq x\:\sqrt{\E |\widehat{Z}_{T,\Delta}(\tau)|^2}\Bigg\}\leq 2K(\sqrt{T}x), \qquad T>0,\quad \Delta>0,
$$
where the function $K$ is as defined in Theorem~\ref{thm3}.

The next theorem can be employed in constructing better-tuned correlogram-based tests for hypotheses related to the form of the cross-correlation function between stationary Gaussian processes.
For this purpose, we will need some notions from Section 3. Recall that for all $T>0$ and $\Delta>0$ the mean-square deviation $\rho_{(T,\Delta)}$ is given by (\ref{F:rho_TDelta}).
Further, let ${\NN}_{\rho_{(T,\Delta)}}(\varepsilon)$ and ${\HH}_{\rho_{(T,\Delta)}}(\varepsilon), \varepsilon>0,$ be the metric massiveness and the metric entropy, correspondingly, of $[a,b]\subset \R$ with respect to $\rho_{(T,\Delta)}$. Choose an arbitrary number $r\in (0,1)$ and set\quad
$
\varepsilon_{T,\Delta}=\left(C_r/\ln2\right)^{1/2} \cdot \Sup_{\tau_1,\tau_2 \in [a,b]} \rho_{(T,\Delta)}(\tau_{1},\tau_{2}),
$\quad
where $C_r=r^{-2}|\ln{(1-r)}|-r^{-1}.$

\begin{thm}\label{thm4}
For any $[a,b]\subset\R$, assume that the following inequality holds\/{\rm :}
\begin{equation}\label{fZTD}
\Int_{0}^{\varepsilon_{T,\Delta}}\ln{(1+{\NN}_{\rho_{(T,\Delta)}}(\varepsilon))}d\varepsilon<\infty.
\end{equation}
Then, for any $x>0$, we have
\begin{equation}\label{fpbZTD}
\emph{\Pb}\Bigg\{\sqrt{T}\Sup_{\tau\in[a,b]}\big|\widehat{H}_{T,\Delta}(\tau)-\emph{\E} \widehat{H}_{T,\Delta}(\tau)\big|\geq x\Bigg\}\leq 2\exp{\Big(-\frac{x}{A_{T,\Delta}}\Big)},
\end{equation}
where
$$
A_{T,\Delta}=\Bigg(\frac{C_r}{\ln2}\Bigg)^{1/2} \Inf_{\tau\in[a,b]}\Big(\emph{\E} |\widehat{Z}_{T,\Delta}(\tau)|^2\Big)^{1/2}+
\Inf_{\theta\in\Theta_{T,\Delta}}\frac{e^2}{\theta(1-\theta)}\Int_{0}^{\theta\varepsilon_{T,\Delta}}\ln{\Bigg(1+{\NN}_{\rho_{(T,\Delta)}}\Big(\varepsilon \Big(\frac{\ln{2}}{C_r}\Big)^{1/2}\Big)\Bigg)}d\varepsilon,
$$
and $\Theta_{T,\Delta}=\{0<\theta<1:\ {\NN}_{\rho_{(T,\Delta)}}(\theta\varepsilon_{T,\Delta})>e^2-1\}$.
\end{thm}

\begin{proof} The proof is based on the fact that $\widehat{Z}_{T,\Delta}$ is a centered square-Gaussian process, being therefore an $\textrm{Exp}_{(1)}$-process. By Theorem 6.2.3 in (Buldygin and Kozachenko 2000),  for all $\tau_1, \tau_2 \in [a,b]$:
$$
\|\widehat{Z}_{T,\Delta}(\tau_2)-\widehat{Z}_{T,\Delta}(\tau_1)\|_{\textrm{Exp}_{(1)}}\leq \Big(\frac{C_r}{\ln{2}}\Big)^{1/2} \sqrt{\Var\big(\widehat{Z}_{T,\Delta}(\tau_2)-\widehat{Z}_{T,\Delta}(\tau_1)\big)}= \Big(\frac{C_r}{\ln{2}}\Big)^{1/2}\rho_{(T,\Delta)}(\tau_1,\tau_2).
$$
Formula (\ref{fpbZTD}) is obtained from the latter upper bound by application of Theorem 3.3.4 in (Buldygin and Kozachenko 2000) to the distribution of the supremum of an $\textrm{Exp}_{(1)}$-process.
\end{proof}

\begin{rem}\label{remZTD1}
For a given $[a,b]\subset\R,$ condition (\ref{fZTD}) in Theorem \ref{thm4} may be replaced by (\ref{f13}). 
 Inequality (\ref{f13}) appears in the proof of Statement\/ {\rm (II)} in Theorem \ref{thm1} dealing with a.~s. continuity of $\widehat{Z}_{T,\Delta}$. \hfill$\triangleleft$
\end{rem}

Observe that, in the class of $\textrm{Exp}_{(1)}$-processes, the rate of decay as $x\to\infty$  appearing on the right-hand side of (\ref{fpbZTD}) is correct. The constant $A_{T,\Delta}$ is overestimated which is typical when entropy methods are applied even to Gaussian processes. The explicit formula for $A_{T,\Delta}$ seems to be rather complex; this quantity may admit simpler assessments in particular cases.

\begin{rem}\label{remZTD2}
Lemma \ref{lem1} implies an upper bound, uniform in $T>0$ and $\Delta>0$, on $\rho_{(T,\Delta)}$:
 $$
 \Sup_{T,\Delta>0}\Sup_{\tau_1,\tau_2\in\R}\rho_{(T,\Delta)}(\tau_1,\tau_2)\leq \frac{2}{c}\:\|H\|_{2}\Big(\Sup_{\Delta>0}\|g^*_{\Delta}\|_{\infty}\Big)
 $$
leading to the inequality
$$
\varepsilon_{T,\Delta}\leq\frac{2}{c}\:\|H\|_{2}\cdot\Big(\Sup_{\Delta>0}\|g^*_{\Delta}\|_{\infty}\Big)\cdot\Bigg(\frac{C_r}{\ln2}\Bigg)^{1/2}.
$$
This enables us to replace $A_{T,\Delta}$ by a constant which is somewhat greater but does not on $T$ or $\Delta$. \hfill$\triangleleft$
\end{rem}

\section{Examples of IRF components}

 \begin{figure}[!h]
\begin{center}
\includegraphics[height=3.7cm]{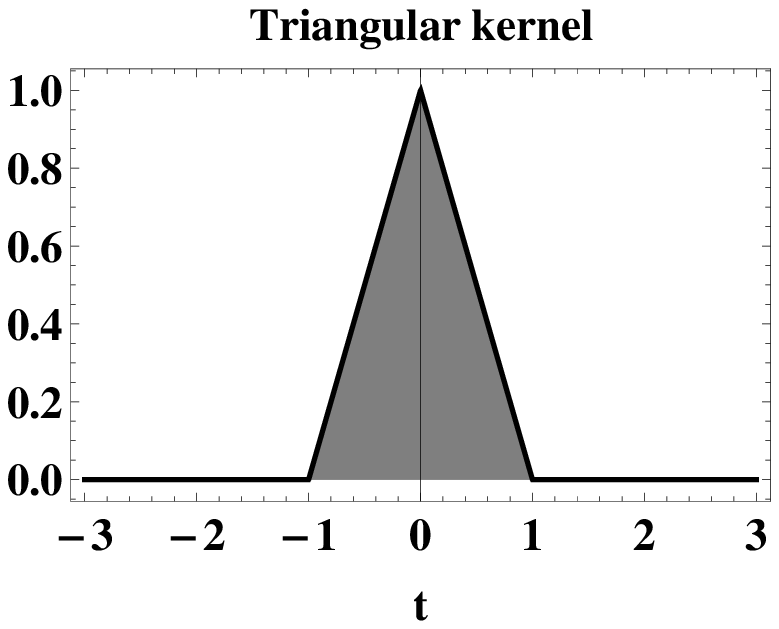}\includegraphics[height=3.7cm]{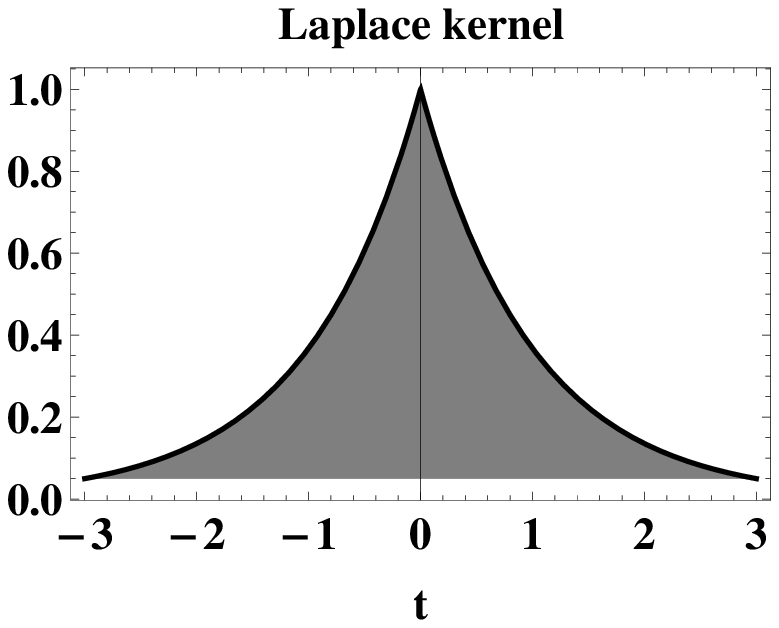}\includegraphics[height=3.7cm]{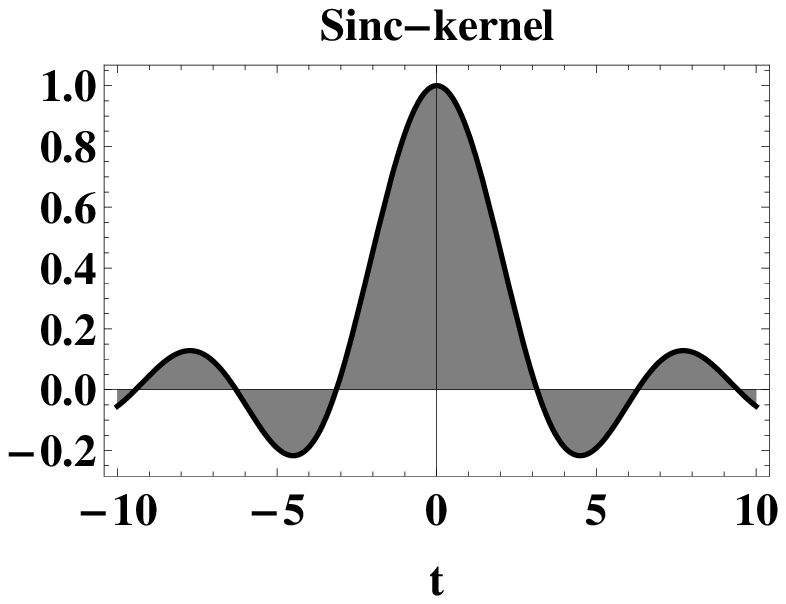}
\end{center}
 \caption{Triangular kernel, Laplace kernel and sinc-kernel.}\label{Fi:3_kernels}
 \end{figure}

We would like to give important particular cases of IRF's $L_2$-integrable components $g_{\Delta}$ and $H$. Three kernels shown in Figure~\ref{Fi:3_kernels} are chosen under the following motivation:
 \begin{itemize}
   \item the triangular kernel is well-known in signal processing as the Bartlett window;
   \item the Laplace kernel is often used for approximation of characteristic functions in harmonic analysis;
   \item the sinc-kernel is an $L_2$-integrable sign-alternating function which appears in (medical) image processing.
 \end{itemize}
For simplicity, we replace convergence of entropy integrals by sufficient conditions that depend on boundedness of weighted integrals constructed with respect to spectral measures only.

Note that the fact that the outputs $X_{\Delta}$ and $Y$ are a.~s. continuous means that $\widehat{H}_{T,\Delta}$ and $\widehat{Z}_{T,\Delta}$ are well-defined as elements of $\Ceabe$; see (\ref{f3}) and (\ref{f4}), respectively. Remark \ref{rem5} is used in checking the functional CLT. In what follows, $\I_{A}(\cdot)$ stands for the indicator function of a set $A\subset \R$.


\subsection{Examples of the controlled component $g_{\Delta}$}\label{subsection3}
By the Fernique theorem (Fernique 1975), a necessary condition for the stationary Gaussian process $X_{\Delta}$ to be a.~s. continuous is that the Dudley integral of $X_\Delta$ be finite: $\Int_{0+}{\HH}_{d_{X_{\Delta}}}^{1/2}(\varepsilon)d\varepsilon<\infty.$ This integral may be bounded from above in terms of the spectral density or the correlation function. In our framework, we will use the already mentioned Hunt condition (e.g., (Cram\'er and Leadbetter 1967, 196)) stating that if for some $\beta>0$ the following integral is finite:
\begin{equation}\label{hunt}
\Int_{-\infty}^{\infty}|g^*_\Delta(\lambda)|^2\ln^{1+\beta}{(1+|\lambda|)}d\lambda<\infty,
\end{equation}
then the process $X_{\Delta}$ is a.~s. continuous on $\R$.

We give an account of functions $g_{\Delta}$ satisfying all above assumptions and appearing in a wide range of engineering problems. Next examples represent the scaled classic window functions with bounded (Example \ref{ex_1}) and unbounded (Example \ref{ex_2}) supports. All simulations of $X_{\Delta}$ which are Wiener shot noise processes were made in Wolfram Mathematica 10.4. For more details on similar kernels and their modelling, see (Roberts at el. 2013) and (Prabhu 2014, Chapter 3).

\begin{expl}[\textbf{Triangular kernel, or the Bartlett window}]\label{ex_1}
For any $\Delta>0$, put
$$
g_{\Delta}(t)=c\Delta\Big(1-\Delta|t|\Big)\I_{[-1,1]}(\Delta t), \quad t\in\R.
$$


\paragraph{Framework conditions.} This is a non-negative even function; its Fourier transform is as follows:
$$
g_{\Delta}^{*}(\lambda)=c\Bigg(\frac{\sin{\big(\frac{\lambda}{2\Delta}\big)}}{\frac{\lambda}{2\Delta}}\Bigg)^2, \quad \lambda\in\R.
$$
It is clear that $\{g_{\Delta},g_{\Delta}^{*}\}\subset L_{p}(\R), p\in [1,\infty]$, and therefore conditions (\ref{f1a})--(\ref{f1c}) are satisfied.
Since $|u^2-v^2|\leq|u-v|\cdot(|u|+|v|)$ and since $|\sin{u}|<|u|,$  we have for some $c>0$ and for any $a>0$
 \begin{eqnarray*}
   \lefteqn{\Lim_{\Delta\to\infty}\Sup_{|\lambda|\leq a}|g^*(\lambda)-c|=c\cdot\Lim_{\Delta\to\infty}\Sup_{|\lambda|\leq a}\Bigg|
\Bigg(\frac{\sin{\big(\frac{\lambda}{2\Delta}\big)}}{\frac{\lambda}{2\Delta}}\Bigg)^2-1\Bigg|\leq
2c\cdot\Lim_{\Delta\to\infty}\Sup_{|\lambda|\leq a}\Big|\frac{\sin{\frac{\lambda}{2\Delta}}-\frac{\lambda}{2\Delta}}{\frac{\lambda}{2\Delta}}\Big|=}\\[2mm]
  &=&
    2c\cdot\Lim_{\Delta\to\infty}\Sup_{|\lambda|\leq a}\Bigg|\frac{\lambda^2}{(2\Delta)^2}\Sum_{k=0}^{\infty}\frac{(-1)^{k}}{(2k+3)!}\Big(\frac{\lambda}{2\Delta}\Big)^{2k}\Bigg|\leq\frac{2c}{3!}\cdot\Lim_{\Delta\to\infty}\Sup_{|\lambda|\leq a}\Bigg|\frac{\lambda^2}{(2\Delta)^2}\Sum_{k=0}^{\infty}\Big(\frac{\lambda}{2\Delta}\Big)^{2k}\Bigg|=\\[2mm]
   &=&
     \frac{c}{3}\cdot\Lim_{\Delta\to\infty}\Sup_{|\lambda|\leq a}\Bigg|\frac{\lambda^2}{(2\Delta)^2}\frac{1}{1-\Big(\frac{\lambda}{2\Delta}\Big)^{2}}\Bigg|=\frac{c}{3}\cdot\Lim_{\Delta\to\infty}\Sup_{|\lambda|\leq a}\frac{\lambda^2}{|(2\Delta)^2-\lambda^2|}\leq \frac{c a^2}{12}\Lim_{\Delta\to\infty}\frac{1}{\Delta^2}=0.
 \end{eqnarray*}
In this way, condition (\ref{f1d}) also holds. The family $(g_{\Delta},\ \Delta>0)$ approximates the Dirac delta as $\Delta\to \infty$.


\paragraph{The structure of $X_{\Delta}$.} The output $X_{\Delta}$ is generated by linear filtration of a standard Wiener process:
$$
X_{\Delta}(t)=c\Delta\Int_{-\infty}^{\infty}\Big(1-\Delta|t-s|\Big)\I_{\{\Delta|t-s|\leq 1\}}(s)d W(s), \quad t\in\R.
$$
   It is Gaussian stationary zero-mean process; its spectral density is $f_{X_{\Delta}}(\lambda)=\frac{1}{2\pi}|g^*_\Delta(\lambda)|^2, \lambda\in\R.$ The $\delta$-like structure of $g_\Delta$ implies that the family $X_{\Delta}$ converges to a Gaussian white noise as $\Delta\to\infty$. This property can be traced down in Figure \ref{Fi:Ex1_simulation} where trajectories of $X_{\Delta}$ are simulated with increasing values of $\Delta=1, 10, 100, 1000$.

 \begin{figure}[!h]
\begin{center}
\includegraphics[height=5.4cm]{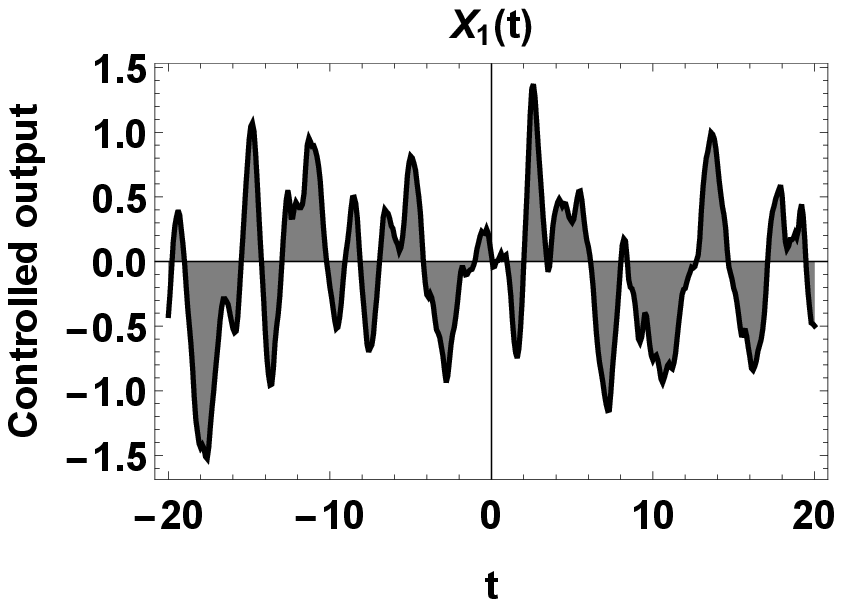}\ \ \ \ \ \ \includegraphics[height=5.4cm]{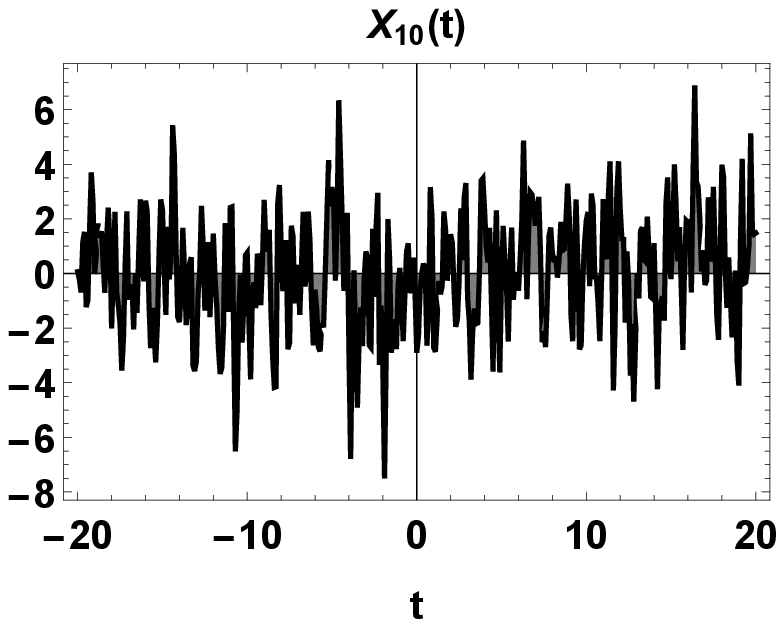}
\includegraphics[height=5.4cm]{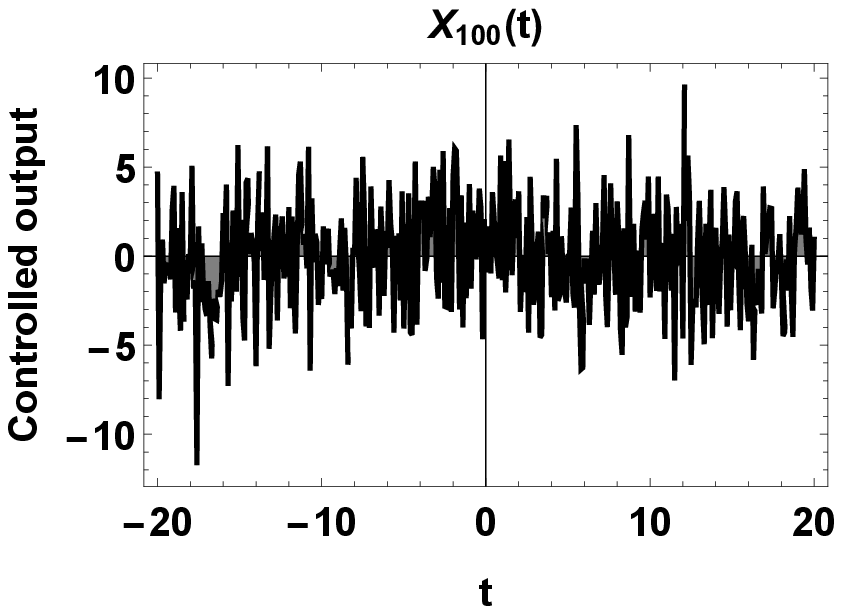}\ \ \ \ \ \ \includegraphics[height=5.4cm]{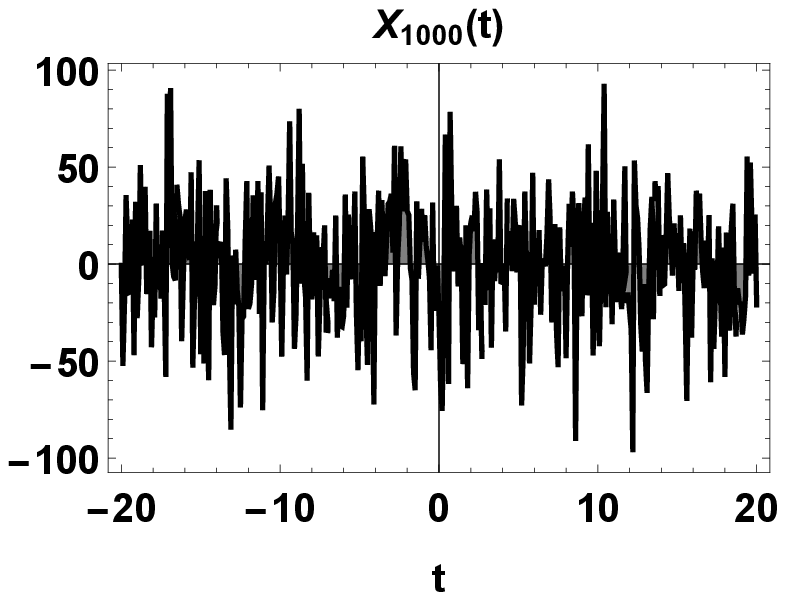}
\end{center}
 \caption{Outputs $X_{\Delta}$ for the triangular kernel $g_{\Delta}$ simulated with $\Delta=1,10,100,1000$; $c=1$.}\label{Fi:Ex1_simulation}
 \end{figure}


\paragraph{Continuity of $X_{\Delta}$.} For any $\Delta>0$ and some $\beta>0$, we have
\begin{eqnarray*}
 \lefteqn{\Int_{-\infty}^{\infty}|g^*_\Delta(\lambda)|^2\ln^{1+\beta}{(1+|\lambda|)}\dd\lambda=c^2\Int_{-\infty}^{\infty}\Bigg(\frac{\sin{\big(\frac{\lambda}{2\Delta}\big)}}{\frac{\lambda}{2\Delta}}\Bigg)^4\ln^{1+\beta}{(1+|\lambda|)}\dd\lambda\asymp}\\[2mm]
 && \asymp2c^2\Big[\Int_{0}^{1}\ln^{1+\beta}{(1+|\lambda|)}\dd\lambda+(-1)^{1+\beta}(2\Delta)^2\Int_{1}^{\infty}\Big(\frac{1}{1+\lambda}\Big)^2\ln^{1+\beta}{\Big(\frac{1}{1+\lambda}\Big)}\dd\lambda\Big]<\infty.
\end{eqnarray*}
The latter implies condition (\ref{hunt}); that is, the controlled process $X_{\Delta}, \Delta>0$ is a.~s. continuous on $\R$.


\paragraph{Engineering interpretation.} The kernel $g_{\Delta}$ is defined as a scaled Bartlett window function. This function often appears in problems of linear interpolation and is widely used in signal processing.
\hfill$\Diamond$
\end{expl}

\begin{expl}[\textbf{Exponential kernel, or the Laplace window}]\label{ex_2}
For any $\Delta>0$, let
$$
g_{\Delta}(t)=\frac{c\Delta}{2}\exp{(-\Delta|t|)}, \quad t\in\R.
$$


\paragraph{Framework conditions.} This function is even and non-negative; its Fourier transform is as follows:
$$
g_{\Delta}^{*}(\lambda)=\frac{c\Delta^2}{\Delta^2+\lambda^2}, \quad \lambda\in\R.
$$
It is clear that $\{g_{\Delta},g_{\Delta}^{*}\}\subset L_{p}(\R),$ $p\in [1,\infty]$, and therefore conditions (\ref{f1a})--(\ref{f1c}) hold.
For any $a>0$
$$
\Lim_{\Delta\to\infty}\Sup_{|\lambda|\leq a}|g^*(\lambda)-c|=c\cdot\Lim_{\Delta\to\infty}\Sup_{|\lambda|\leq a}\frac{\lambda^2}{\Delta^2+\lambda^2}=c\cdot\Lim_{\Delta\to\infty}\frac{a^2}{\Delta^2}=0
$$
implying (\ref{f1d}). The family $(g_{\Delta}, \Delta>0)$ satisfies (\ref{f1a})--(\ref{f1d}) and approximates the Dirac delta as $\Delta\to \infty$.


\paragraph{The structure of $X_{\Delta}$.} The output $X_{\Delta}$ is generated by linear filtration of a standard Wiener process:
$$
X_{\Delta}(t)=\frac{c\Delta}{2}\Int_{-\infty}^{\infty}\exp{(-\Delta|t-s|)}\dd W(s), \quad t\in\R.
$$
In some sense, the last formula is a generalised solution to a Langevin equation describing particle's velocity in a fluid with unlimited time. It is the well-known Ornstein-Uhlenbeck process on $\R$. 
That is, we deal with a Gaussian stationary zero-mean process with spectral density $f_{X_{\Delta}}(\lambda)=\frac{1}{2\pi}|g^*_\Delta(\lambda)|^2,$ $\lambda\in\R.$ As we already know, the family $X_{\Delta}$ converges to a Gaussian white noise as $\Delta\to\infty$. This fact can be observed in Figure \ref{Fi:Ex2_simulation} where trajectories of $X_{\Delta}$ are simulated with increasing values of $\Delta=1,10,100,1000$.

 \begin{figure}[!h]
\begin{center}
\includegraphics[height=5.4cm]{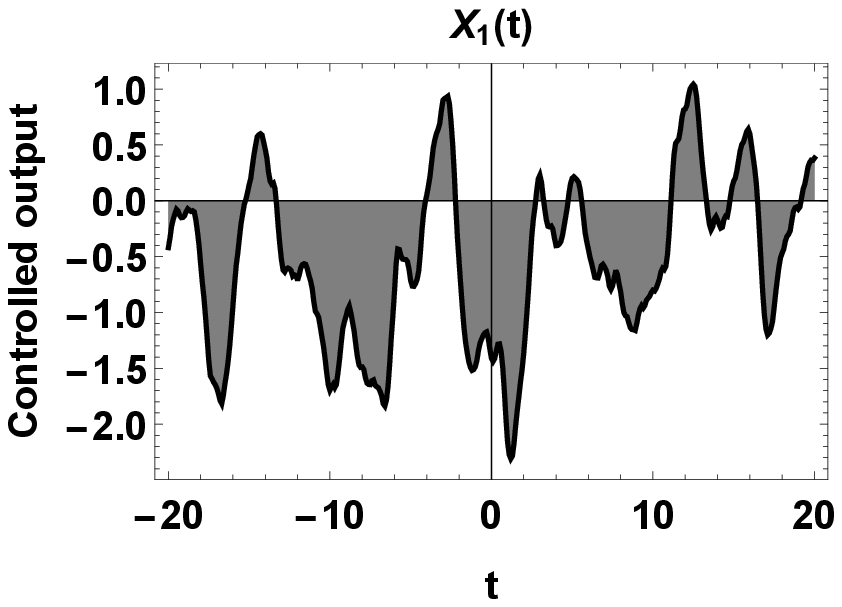}\ \ \ \ \ \ \includegraphics[height=5.4cm]{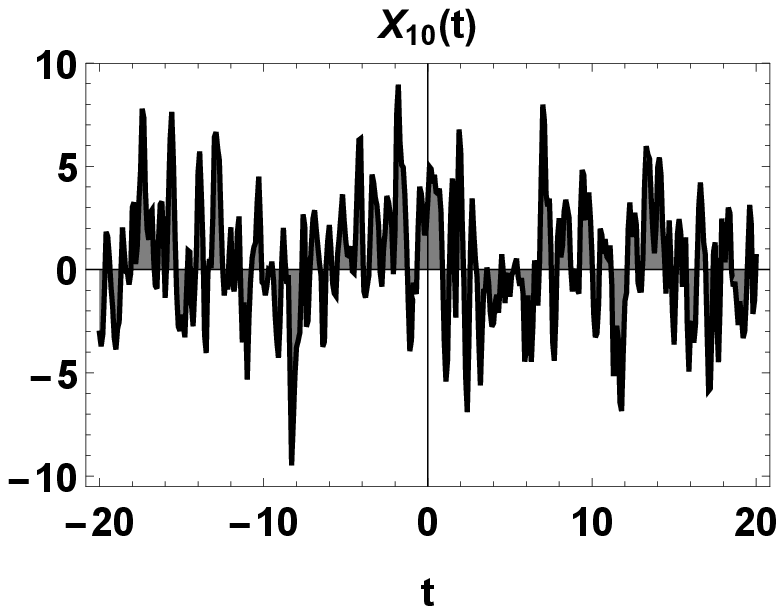}
\includegraphics[height=5.4cm]{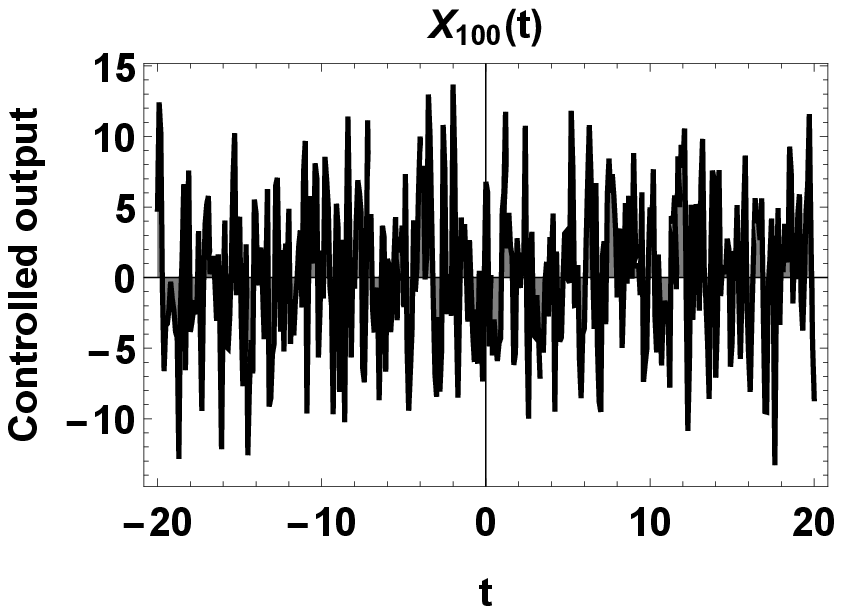}\ \ \ \ \ \includegraphics[height=5.4cm]{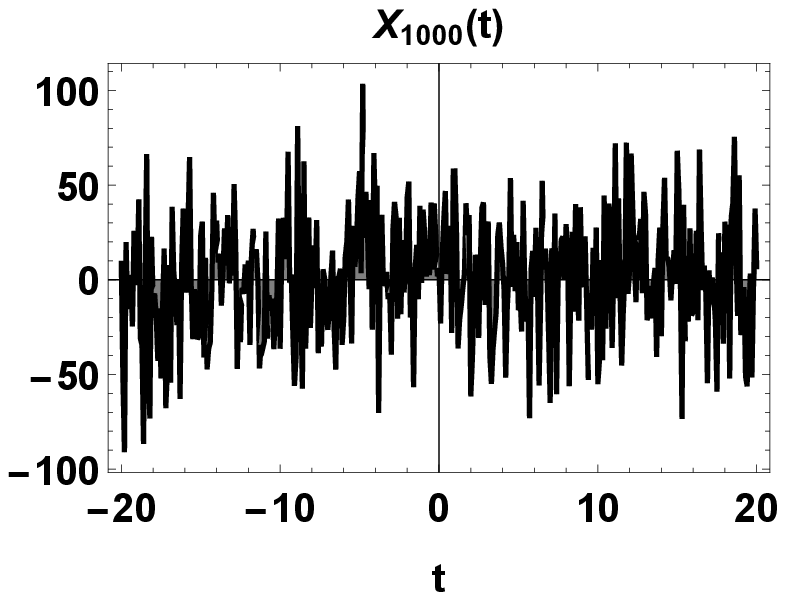}
\end{center}
 \caption{Outputs $X_{\Delta}$ for the exponential kernel $g_{\Delta}$ simulated with $\Delta=1,10,100,1000$; $c=2$.}\label{Fi:Ex2_simulation}
 \end{figure}

\noindent \textbf{Continuity of $X_{\Delta}$.} For any $\Delta>0$ and some $\beta \in(0,2)$, we have
$$
\Int_{-\infty}^{\infty}|g^*_\Delta(\lambda)|^2\ln^{1+\beta}{(1+|\lambda|)}\dd\lambda=\Int_{-\infty}^{\infty}\Big(\frac{c\Delta^2}{\Delta^2+\lambda^2}\Big)^2\ln^{1+\beta}{(1+|\lambda|)}\dd\lambda \leq 2c^2\Big[(\ln{2})^{4+\beta}+\Int_{1}^{\infty}\frac{\dd\lambda}{\lambda^{3-\beta}}\Big]<\infty.
$$
 The latter implies (\ref{hunt}); that is, the controlled output $X_{\Delta}$ is a.~s. continuous on $\R$.


\paragraph{Probability interpretation.} The kernel $g_{\Delta}$ is defined as a normalised (by $c$) probability density function of the Laplace distribution located at zero and whose scale parameter is $\Delta>0$. The component $g_{\Delta}$ may also be interpreted as a normalised (by $c\Delta/2$) characteristic function of the standard Cauchy distribution with scale parameter $\Delta$. Let us only say that the type of scaling is non-random and is related to harmonic analysis and convolutions involving distribution densities. In some sense, our study uses the inverse Fourier transform for convolutions without assuming absolute integrability of kernels. \hfill$\Diamond$
\end{expl}

\subsection{Examples of the unknown component $H$}\label{subsection4}
Recall that, in our framework, the component $H$ appears in the functional CLT, in the structure of entropy characteristics and in the construction of confidence intervals. To be more precise:
\begin{itemize}
  \item $H\in L_2(\R)$ which is sufficient for CLT of the centred estimator;
  \item for some $\beta>0$, we have:
  \begin{equation}\label{CLT}
  \Int\limits_{-\infty}^{\infty}|H^{*}(\lambda)|^{2}\log^{4+\beta}(1+|\lambda|)\,d\lambda<\infty
  \end{equation}
  which is sufficient for obtaining a CLT (Remark \ref{rem5} and Theorem \ref{thm1}) and for constructing confidence intervals in $C[a,b]$ (Theorem \ref{thm2} and Theorem \ref{thm3}).
  \end{itemize}
Further two examples of IRF kernels $H$ are useful in signal and image processing. These examples cover the cases of even and odd $H$; both functions are defined by continuity at zero.

\begin{expl}[\textbf{Sinc-kernel, or the Shannon scale function}]\label{ex_3}
Consider the so-called normalised sinc-function
$$
H(t)=\mathop{sinc}(t)= \frac{\sin \pi t}{\pi t}, \quad t\in \R.
$$
It is clear that $H\in L_2(\R)$; the Fourier--Plancherel transform of this function is well-known to be
$$
H^*(\lambda)=\Int_{-\infty}^{\infty}e^{-i\lambda t}\Big(\frac{\sin \pi t}{\pi t}\Big) dt = \I_{[-\pi,\pi]}(\lambda),\quad \lambda\in\R.
$$
It is clear that $H\in L_p(\R), p\in(1,\infty]$ and $H^* \in L_q(\R), q\in[1,\infty]$.
For any $\beta>0$, we have
$$
\Int_{-\infty}^{\infty}|H^{*}(\lambda)|^{2}\log^{4+\beta}(1+|\lambda|)\,d\lambda=2\Int_{0}^{\pi}\log^{4+\beta}(1+\lambda)\,d\lambda<\infty.
$$
Therefore $H$ satisfies all our assumptions and gives an example of an ideal low-pass filter.

Haar functions and sinc-functions are Fourier duals of each other; it is easy to make calculations in our case. Since $H$ is even, $H^*$ is real-valued: $H^*=\overline{H^*}$. Thus, the correlation function of $Y$ is as follows:
$$
\vspace{-2mm}
K_{Y}(t)=\E Y(t)Y(0)=\frac{1}{2\pi}\Int_{-\infty}^{\infty}e^{it\lambda}|H^*(\lambda)|^2\dd\lambda=\frac{\sin{(\pi t)}}{\pi t}=H(t),\quad t\in\R,
$$
and the correlation function of $Z$ has the form
 \begin{eqnarray*}
   C_{\infty}(\tau_1,\tau_2) &= & \E Z(\tau_1)Z(\tau_2)=\frac{1}{2\pi}\Int_{-\infty}^{\infty}\Big[e^{i(\tau_1-\tau_2)\lambda}|H^*(\lambda)|^2\dd\lambda+e^{i(\tau_1+\tau_2)\lambda}\big(H^*(\lambda)\big)^2\dd\lambda\Big]= \\[2mm]
   &=& K_Y(\tau_1-\tau_2)+K_Y(\tau_1+\tau_2)=H(\tau_1-\tau_2)+H(\tau_1+\tau_2),\quad \tau_1,\tau_2 \in\R.
 \end{eqnarray*}
These correlation functions of zero-mean Gaussian processes $Y$ and $Z$ are based on the only assumption that $H$ is $L_2$-integrable. The function $C_{\infty}$ appears to be the sum of two sinc-waves, one of which is responsible for stationarity (main diagonal) and the second wave ``controls'' non-stationarity (anti-diagonal). Recall also that in this case $\big(H\ast H\big)(\tau)=H(\tau),$ $\tau \in \R.$ \hfill$\Diamond$
\end{expl}

The function $H$ in Example \ref{ex_3}  is the scaling function for the Shannon MRA (signal analysis by ideal bandpass filters). The expression of a real-valued Shannon wavelet can be obtained by taking $\psi^{\textrm{\scriptsize Sha}}(t)=2H(2t-1)-H(t),$ $t\in \R.$ This wavelet belongs to the $C^{\infty}$-class, but it decreases slowly at infinity and has no bounded support. For a detailed information on Shannon wavelets and their applications, we refer the reader to the book (Najmi 2012, Chapters 4--5).

Our last example shows an odd $L_2$-integrable function which is a Hilbert transform of the sinc-kernel already mentioned in Example \ref{ex_3}. In some sense, we save the duality between Haar- and sinc-systems, and show how the fact that $H$ is odd affects the characteristics of $Y$ and $Z$.

\begin{expl}[\textbf{Hilbert transform of the Sinc-kernel}]\label{ex_4}
Take
$$
H(t)= \frac{1-\cos{(\pi t)}}{\pi t}, \quad t\in \R.
$$
Of course, one has $H\in L_2(\R)$; the Fourier--Plancherel transform of $H$ is as follows:
$$
H^*(\lambda)=\Int_{-\infty}^{\infty}e^{-i\lambda t}\Big(\frac{1-\cos{(\pi t)}}{\pi t}\Big) dt = i\cdot \mathop{sign}(\lambda)\cdot\I_{[-\pi,\pi]}(\lambda),\quad \lambda\in\R.
$$
It is clear that $H\in L_p(\R),$ $p\in(1,\infty]$ and $H^* \in L_q(\R), q\in[1,\infty].$
For any $\beta>0,$ we have
$$
\Int_{-\infty}^{\infty}|H^{*}(\lambda)|^{2}\log^{4+\beta}(1+|\lambda|)\,d\lambda=2\Int_{0}^{\pi}\log^{4+\beta}(1+\lambda)\,d\lambda<\infty.
$$
This means that $H$ satisfies all required assumptions. It is a Hilbert transform of an ideal low-pass filter.
The fact that $H$ is odd implies that $H^*$ is imaginary-valued and $H^*=-\overline{H^*}$. Then the correlation function of $Y$ is as follows:
$$
K_{Y}(t)=\E Y(t)Y(0)=\frac{1}{2\pi}\Int_{-\infty}^{\infty}e^{it\lambda}|H^*(\lambda)|^2\dd\lambda=\frac{\sin{(\pi t)}}{\pi t},\quad t\in\R,
$$
and the correlation function of $Z$ is:
 \begin{eqnarray*}
   C_{\infty}(\tau_1,\tau_2) &=&
    \E Z(\tau_1)Z(\tau_2)=\frac{1}{2\pi}\Int_{-\infty}^{\infty}\Big[e^{i(\tau_1-\tau_2)\lambda}|H^*(\lambda)|^2\dd\lambda+e^{i(\tau_1+\tau_2)\lambda}\big(H^*(\lambda)\big)^2\dd\lambda\Big]=\\[2mm]
   &=& K_Y(\tau_1-\tau_2)-K_Y(\tau_1+\tau_2)=H(\tau_1-\tau_2)-H(\tau_1+\tau_2),\quad \tau_1,\tau_2 \in\R.
 \end{eqnarray*}
Again the correlation functions of $Y$ and of $Z$ are constructed under the only restriction that $H$ is $L_2$-integrable. The correlation function of the non-stationary process $Z$ splits into two sinc-waves; its shape is quite interesting since $C_{\infty}(0,0)=0$. \hfill$\Diamond$
\end{expl}

It follows from Examples \ref{ex_3} and \ref{ex_4} that the correlation function of the separable Gaussian stationary zero-mean process $Y$ has the form: $K_{Y}(t)=\mathop{sinc}(t),$ $t\in\R.$ Then
$K_{Y}(t)=1-\pi^2t^2/3!+o(t^2)$, $t\to 0$,
and we can estimate the supremum of $Y$ using general results in (Cram\'er and Leadbetter 1967, Chapters 9--10), (Lifshits 1995, Chapters 12--14) or (Adler 2000, Chapter V).

\section{Concluding remarks}
We gave the principle of estimation the kernels of Wiener shot-noise processes. For this purpose, the LTI SIDO-models were investigated with the Wiener signal as the input. Assuming that one IRF's component has been controlled, we used the normalized cross-correlogram between outputs as an estimator for the unknown second one. According to the space we chosen, the origin $L_2$-integrability condition on IRF's components was reinforced by weighed integrals which were dependent on FRFs of both kernels.

\paragraph{Acknowlegements.} This research was partially supported by the MINECO/FEDER under Grant MTM2015-69493-R. The first author would like to thank the Research Group on Advanced Statistical Modelling (coordinators Prof P.~Puig-Casado and J.~del Castillo-Franquet) for hospitality during her research stay at the Universitat Aut\`onoma de Barcelona.


\end{document}